\documentclass[a4paper,11pt]{article}
\usepackage{pdfsync}

\usepackage[plainpages=false]{hyperref}
\usepackage{amsfonts,amsmath,amssymb,amsthm,enumerate}
\usepackage{latexsym,lscape,rawfonts,mathrsfs,mathtools}
\usepackage{enumitem,pifont}
\usepackage{cases}
\usepackage{stmaryrd,accents}

\usepackage[all]{xy}
\usepackage{makeidx}
\usepackage{graphicx,psfrag}
\usepackage{pstool}
\usepackage{tikz-cd}
\usepackage{array,tabularx,tabu}

\usepackage{setspace}

\usepackage[titletoc,title]{appendix}

\usepackage{txfonts}

\newcommand{\C}{\ensuremath{\mathbb{C}}}

\newcommand{\R}{\ensuremath{\mathbb{R}}}
\newcommand{\T}{\ensuremath{\mathbb{T}}}
\newcommand{\Z}{\ensuremath{\mathbb{Z}}}

\newcommand{\ba}{\begin{align*}}
\newcommand{\ea}{\end{align*}}
\newcommand{\na}{\nabla}
\newcommand{\la}{\langle}
\newcommand{\ra}{\rangle}
\newcommand{\lc}{\left(}
\newcommand{\rc}{\right)}

\newcommand{\ep}{\epsilon}

\newcommand{\Rc}{\mathrm{Rc}}

\newcommand{\g}{\mathfrak{g}}
\newcommand{\h}{\mathfrak{h}}
\newcommand{\lt}{\mathbf{\mathfrak{t}}}

\newcommand{\LL}{\mathcal{L}}

\makeatletter
\def\ExtendSymbol#1#2#3#4#5{\ext@arrow 0099{\arrowfill@#1#2#3}{#4}{#5}}

\makeatother

\makeatletter
\def\ExtendSymbol#1#2#3#4#5{\ext@arrow 0099{\arrowfill@#1#2#3}{#4}{#5}}

\makeatother

\def\XXint#1#2#3{{\setbox0=\hbox{$#1{#2#3}{\int}$ }
\vcenter{\hbox{$#2#3$ }}\kern-.55\wd0}}

\numberwithin{equation}{section}

\newtheorem{thm}{Theorem}[section]
\newtheorem{cor}[thm]{Corollary}
\newtheorem{prop}[thm]{Proposition}
\newtheorem{lem}[thm]{Lemma}

\newtheorem{rem}[thm]{Remark}

\newtheorem{defn}[thm]{Definition}

\newtheorem{exmp}[thm]{Example}

\title{Uniqueness of K\"ahler Ricci shrinkers on toric orbifolds}
\author{Yu Li \quad and \quad Wenjia Zhang}
\date{\today}

\begin{document}
\maketitle

\begin{abstract}
In this paper, we prove the uniqueness of K\"ahler Ricci shrinkers on toric orbifolds, extending the corresponding results previously established for toric manifolds.
\end{abstract}

\tableofcontents

\section{Introduction}

A Ricci shrinker $(M, g, f)$ on a manifold or orbifold $M$ is a complete Riemannian manifold or orbifold coupled with a smooth function $f:M \to \R$ such that
\begin{align*}
\Rc+\na^2 f=\frac{1}{2} g.
\end{align*} 

Ricci shrinkers on manifolds have been a central topic of study in the theory of Ricci flow. They are significant for two main reasons. First, they can be regarded as canonical metrics on manifolds, serving as natural generalizations of Einstein metrics. While Einstein metrics are static solutions to the Ricci flow, Ricci shrinkers represent self-similar solutions. Second, as critical metrics of Perelman's celebrated $\mu$-functionals, Ricci shrinkers play a crucial role in the singularity analysis of Ricci flow. They provide models for understanding the formation of singularities, allowing for the classification and resolution of those singular regions.

In dimensions 2 and 3, all Ricci shrinkers on manifolds have been fully classified; see \cite{Ha95} \cite{Pe02} \cite{Naber} \cite{NW} \cite{CCZ}, among others. The complete list consists of $\R^2,S^2,\R^3,S^3,S^2 \times \R$ and their quotients. A K\"ahler Ricci shrinker $(M, g, J, f)$ is a Ricci shrinker on a complex manifold $(M, J)$ so that $g$ is a K\"ahler metric. The classification of K\"ahler Ricci shrinkers is crucial for understanding singularity formation along the Kähler Ricci flow. Recently, the classification of all K\"ahler Ricci shrinkers on complex surfaces has been fully completed (cf. \cite{CDS24} \cite{CCD24} \cite{BCCD24} and \cite{LW23}). However, in higher dimensions, the classification of Ricci shrinkers remains an open and challenging problem that seems beyond current reach.

The uniqueness of Ricci shrinkers within certain classes is one of the most important questions in the study of Ricci shrinkers. On compact complex manifolds, Tian and Zhu \cite{TZ02a} \cite{TZ02b} have established a general uniqueness theorem for Kähler Ricci shrinkers in the anti-canonical class. However, for noncompact cases, the situation is far more complicated and remains unresolved in general. In the specific case of complex toric manifolds, the uniqueness of K\"ahler Ricci shrinkers was proven by Cifarelli \cite{Ci22}, offering progress in this particular setting.

Orbifolds are natural generalizations of manifolds that allow for singularities and are locally modeled on quotients of $\R^n$ by a finite group action. They arise naturally in various areas, such as differential geometry and algebraic geometry. Ricci flows and Ricci shrinkers on orbifolds have been studied by many people. For instance, it was shown in \cite{Wu91} that there exists a unique Ricci shrinker on any $2$-dimensional bad orbifold without boundary. Another notable result is that the application of Ricci flow on orbifolds provides a new proof of the geometrization of $3$-dimensional orbifolds \cite{KL14}. Even in the context of Ricci flow on manifolds, Ricci shrinkers on orbifolds emerge naturally as models for singularities (cf. \cite{CW12} \cite{Bam20}).

In this paper, we extend the aforementioned uniqueness result to toric orbifolds. Our main theorem is:

\begin{thm}
\label{thm:A}
Suppose $(M^{2n}, J, \T^n)$ is a complex orbifold with an effective and holomorphic $\T^n$-action. Then, up to equivariant  biholomorphism, there is at most one $\T^n$-invariant K\"ahler Ricci shrinker $(M,g, J, f)$ with bounded scalar curvature.
\end{thm}

In the setting of Theorem \ref{thm:A}, if a K\"ahler Ricci shrinker $(M,g, J, f)$ exists, an important observation is that $(M^{2n}, \omega, \T^n)$, where $\omega$ is the corresponding K\"ahler form, is a symplectic toric orbifold with a Hamiltonian $\T^n$-action (cf. Proposition \ref{prop:kros} (ii)). Furthermore, the holomorphic Killing field $J\na f$ belongs to $\lt$, the Lie algebra of $\T^n$ (cf. Proposition \ref{prop:kros} (iii)). To ensure this, it is necessary to assume that the scalar curvature of $g$ is globally bounded below, a condition that automatically holds in the manifold case.

Symplectic toric orbifolds, like symplectic toric manifolds, have been well understood for compact cases. Lerman and Tolman \cite{LT97} demonstrated that a compact symplectic toric orbifold corresponds one-to-one with a labeled, rational, simple polytope (cf. Definition \ref{def:poly}), which is the image of the moment map. The local structure groups of the orbifold can be determined from the integers attached to the facets of the polytope. However, the noncompact case is far more complicated. Without additional conditions, the image of the moment map may not even be convex (cf. \cite{Pr94}). In our case, we prove that the image of the moment map of $(M^{2n}, \omega, \T^n)$ is also a proper, rational, simple polyhedron $P$ in $\lt^*$, with an integer attached to each facet. A key point is that the potential function $f$, which is a Hamiltonian function of $J \na f$, is proper and bounded below. Additionally, the boundedness of the scalar curvature ensures that the set of critical points of $f$ must be compact, implying that there are only finitely many fixed points of the $\T^n$-action. In fact, we establish a one-to-one correspondence between all such toric symplectic orbifolds and their associated labeled polyhedra; see Theorem \ref{T204} and Theorem \ref{T206}.

In particular, there exists a standard symplectic toric orbifold $(M_P, \omega_P ,J_P)$ from Delzant's construction, which is also a toric variety (see Definition \ref{def:tr}). On the one hand, there exists a $\T^n$-invariant symplectomorphism between $(M, \omega)$ and $(M_P, \omega_P)$. On the other hand, under our assumptions, the $\T^n$-action extends uniquely to a $(\C^*)^n$-action. By using the holomorphic slice theorem (cf. Theorem \ref{thm:hslice}), one can demonstrate that $(M, J, (\C^*)^n)$ is also a toric variety, biholomorphic to $(M_P, J_P)$ (cf. Theorem \ref{thm:to}). Moreover, the labeled polyhedron $P$ depends only on the class of $[\omega]$ in the orbifold de Rham cohomology $H^2_{\mathrm{dR}}(M)$; see Corollary \ref{cor:class}.

In our case, the labeled polyhedron $P$ is determined by the anti-canonical class. Using the renowned Abreu-Guillemin theory, the Ricci shrinker equation on the principal orbit of $(\C^*)^n$ can be transformed into a Monge-Amp\`ere-type equation on the interior of $P$, regarding the corresponding symplectic potential (cf. Proposition \ref{prop:symp}), with appropriate boundary conditions (cf. Theorem \ref{thm:bound}). The equation of the symplectic potential is determined, as the holomorphic Killing field $J\na f$ is independent of the metric $g$. This fact is well-known in the compact case, where it is established using the weighted Futaki invariant (cf. \cite[Theorem 1.4]{SZ11}). Following the methods in \cite{CDS24} and \cite{Ci22}, we show that $J \na f$ is the unique minimum of the weighted volume functional; see Definition \ref{def:wvf}.

Since the metric $g$ is assumed to be complete, the uniqueness of the K\"ahler Ricci shrinker in our case reduces to proving the uniqueness of the solution to the equation regarding the symplectic potential. This follows from the convexity and rigidity of the Ding functional along geodesics in the space of symplectic potentials \cite{BB13} \cite{Ci22}. The convexity ensures that the critical point, which corresponds to a solution, must be unique if it exists. Moreover, the rigidity implies that any two such solutions can differ only by the addition of a linear function. These facts together confirm the uniqueness of the K\"ahler Ricci shrinker, thereby completing the proof of Theorem \ref{thm:A}.

If we do not assume that the K\"ahler Ricci shrinker is $\T^n$-invariant, but instead impose the stronger assumption of bounded Ricci curvature, we have the following result, which is also known in the manifold case \cite[Theorem B]{Ci22}.

\begin{thm}
\label{thm:B}
Suppose $(M^{2n}, J, \T^n)$ is a complex orbifold with an effective and holomorphic $\T^n$-action. Then, up to biholomorphism, there is at most one K\"ahler Ricci shrinker $(M,g, J,f)$ with $J\na f \in \lt$ and bounded Ricci curvature.
\end{thm} 

The key element in the proof of Theorem \ref{thm:B} is a Matsushima-type theorem, originally proved for K\"ahler Ricci shrinkers on manifolds \cite[Theorem 5.1]{CDS24}. This theorem is the primary reason why the stronger assumption of bounded Ricci curvature is required. This result can be directly extended to the orbifold setting with nearly the same proof. With this in place, any K\"ahler Ricci shrinker in the context of Theorem \ref{thm:B} can be made $\T^n$-invariant after applying a biholomorphism of $M$. Therefore, the uniqueness of the shrinker follows from Theorem \ref{thm:A}.

The organization of this paper is as follows. In Section 2, we review fundamental concepts and properties of orbifolds. Additionally, we establish a one-to-one correspondence between a class of symplectic toric orbifolds and labeled polyhedra. We also revisit Delzant's construction and demonstrate that symplectic toric orbifolds in our case are toric varieties. In Section 3, we review the Abreu-Guillemin theory for K\"ahler metrics on toric orbifolds and present some key estimates for Ricci shrinkers on orbifolds. Finally, in Section 4, we prove our main theorems using the results and concepts developed in the preceding sections.
\\
\\

Acknowledgements: Yu Li is supported by YSBR-001, NSFC-12201597 and research funds from the University of Science and Technology of China and the Chinese Academy of Sciences. Both authors would like to thank Prof. Jiyuan Han for helpful discussions.

\section{Symplectic toric orbifolds}

\subsection{Geometry of orbifolds}

The concept of orbifolds was first introduced by Satake in \cite{Sa56} as $V$-manifolds. We begin by reviewing the fundamental definitions of orbifolds and the associated differential geometric concepts. The presentation of these definitions is primarily drawn from sources including \cite{Bor92} \cite{BMP03} \cite{C19} \cite{KL14} and \cite{LT97}.

\begin{defn}[Vector orbispace] 
A vector orbispace is a pair $(V, \Gamma)$, where $V$ is a vector space and $\Gamma$ is a finite group that acts linearly and effectively on $V$. A \emph{linear map} between vector orbispaces $(V, \Gamma)$ and $(V', \Gamma')$ consists of a linear map $T: V \to V'$ and a homomorphism $h: \Gamma \to \Gamma'$ such that for any $g \in \Gamma$ and $v \in V$, the relation $h(g)(Tv) = T(gv)$ holds. 
\end{defn}

\begin{defn}[Local chart] An $n$-dimensional local chart $(U_p, \tilde{U}_p, \Gamma_p, \pi_p)$ around a point $p$ in a topological space $X$ consists of: \begin{enumerate} \item An open neighborhood $U_p$ of $p$ in $X$; \item An open connected subset $\tilde{U}_p$ of $\mathbb{R}^n$, where $\Gamma_p$ is a finite group acting smoothly and effectively on $\tilde{U}_p$; \item A $\Gamma_p$-equivariant projection $\pi_p: \tilde{U}_p \to U_p$, where the action is trivial on $U_p$, inducing a homeomorphism $\tilde{U}_p/\Gamma_p \to U_p$. \end{enumerate} \end{defn}

By slight abuse of notation, a local chart $(U, \tilde{U}, \Gamma, \pi)$ is often denoted simply by $(\tilde{U}, \Gamma)$.

\begin{defn}[Smooth orbifolds] An $n$-dimensional smooth orbifold $M$ consists of a connected, Hausdorff, paracompact topological space $|M|$, known as the underlying space, along with a maximal collection of $n$-dimensional local charts $(\tilde U_i, \Gamma_i)$ such that: \begin{enumerate} \item The collection ${U_i}$ forms an open cover of $|M|$; \item For any $p \in U_1 \cap U_2$, there exists a local chart $U_3$ with $p \in U_3 \subset U_1 \cap U_2$ and embeddings $(\tilde{U}_3, \Gamma_3) \to (\tilde{U}_1, \Gamma_1)$ and $(\tilde{U}_3, \Gamma_3) \to (\tilde{U}_2, \Gamma_2)$. \end{enumerate} \end{defn}

An embedding $(\tilde{U}, \Gamma) \to (\tilde{U}', \Gamma')$ between two local charts is defined as a smooth embedding $\tilde{\phi}: \tilde{U} \to \tilde{U}'$, together with a group homomorphism $\bar{\phi}: \Gamma \to \Gamma'$ such that $\tilde{\phi}$ is $\bar{\phi}$-equivariant.

Given a smooth orbifold $M$ and any $p \in |M|$, a local chart $(U_p, \tilde{U}_p, \Gamma_p, \pi_p)$ around $p$ is called \emph{good} if $\pi_p^{-1}(p)$ consists of a single point $\tilde{p}$. In this case, $\Gamma_p$ is referred to as the \emph{structure group} at $p$. It is straightforward to see that a good local chart around $p$ can always be found, and the structure group at $p$ is independent of the choice of local chart. Without loss of generality, we will consider only good local charts. The tangent space $T_{\tilde{p}} \tilde{U}_p$ is called the \emph{uniformized tangent space} and is denoted by $\tilde{T}_p M$. The \emph{tangent space} at $p$, denoted by $T_p M$, is defined to be the vector orbispace $(\tilde{T}_p M, \Gamma_p)$, where the group action of $\Gamma_p$ on $\tilde{T}_p M$ is induced by the differential.

The \emph{regular part} $M_{\mathrm{reg}} \subset |M|$ consists of points with a trivial structure group. It is clear that $M_{\mathrm{reg}}$ forms a smooth manifold and constitutes an open, dense set of $|M|$. The \emph{singular part} of $M$ is defined as the set $|M|_{\mathrm{sing}}=|M| \backslash M_{\mathrm{reg}}$. We denote by $\Sigma_{\Gamma}$ the subset of $|M|$ consisting of points with the structure group $\Gamma$. This leads to the canonical stratification
\begin{equation}\label{eq:decom}
|M|=\coprod_{\alpha} \Sigma_{\alpha},
\end{equation}
where each $\Sigma_{\alpha}$, called a \emph{stratum}, is a connected component of some $\Sigma_{\Gamma}$.

Next, we recall the notion of smooth maps between orbifolds.
\begin{defn}[Smooth map]
Let $M$ and $N$ be smooth orbifolds. A smooth map $f:M \to N$ is a continuous map $|f|: |M| \to |N|$ that satisfies the following condition:

For any $x \in |M|$ there exist local charts $(U_x, \tilde U_x, \Gamma_x, \pi_x)$ around $x$ and $(V_{y}, \tilde V_{y}, \Gamma_{y}, \pi_{y})$ around $y=|f|(x)$ such that $f(U_x) \subset V_{y}$ and there exists a smooth local lift $\tilde{f}_x : \tilde{U}_x \rightarrow \tilde{V}_{y}$, along with a homomorphism $\bar{f}_x : \Gamma_x \rightarrow \Gamma_{y}$ making the following diagram commute:
\begin{equation*}
\begin{tikzcd}
\tilde{U}_x \arrow[r, "\tilde{f_x}"] \arrow[d, "\pi_x"] & \tilde{V}_{y} \arrow[d, "\pi_{y}"] \\
U_x \arrow[r, "|f|"] & V_{y} 
\end{tikzcd}
\end{equation*}
Additionally, $\tilde f_x$ is $\bar f_x$-equivariant, meaning that for any $t \in \Gamma_x$ and $z \in \tilde{U}_x$, we have $\tilde{f}_x(tz)=\bar f_x(t) \tilde f_x(z)$.

The map $f$ is called a diffeomorphism if it admits a smooth inverse. In this case, $\Gamma_x$ is isomorphic to $\Gamma_{|f|(x)}$.
\end{defn}

For a smooth map $f : M \rightarrow N$ and any $x \in |M|$ with $y=|f|(x)$, a local lift $\tilde{f}_x: (\tilde{U}_x, \Gamma_x, \pi_x) \rightarrow (\tilde{V}_{y}, \Gamma_{y}, \pi_{y})$ has an $\bar f_x$-equivariant differential $d \tilde{f}_x : T_{\tilde x} \tilde{U}_x \rightarrow T_{\tilde{y}} \tilde{V}_{y}$, which induces a linear map $d f_x : T_x M \rightarrow T_{y} N$. This map $df_x$ is called the \emph{differential} of $f$. We say that $f$ is an \emph{immersion} (resp. \emph{submersion}) at $x$ if the map $df_x$ is injective (resp. surjective). The map $f$ is called an immersion (resp. submersion) if $f$ is an immersion (resp. submersion) at each point of $|M|$.

\begin{defn}[Suborbifold]
An orbifold $S$ is a \emph{suborbifold} of a smooth orbifold $M$ if there exists an immersion $i : S \rightarrow M$ such that $|i|$ maps $|S|$ homeomorphically onto its image in $|M|$. Moreover, for each $x \in |S|$, there exists a local chart $(U_{|i|(x)},\tilde U_{|i|(x)},\Gamma_{|i|(x)}, \pi_{|i|(x)} )$ around $|i|(x)$ such that $\pi^{-1}_{|i|(x)}(|i|(S))$ is a submanifold of $\tilde U_{|i|(x)}$.
\end{defn}

Note that $\bar{i}_x : \Gamma_x \rightarrow \Gamma_{|i|(x)}$ is injective for each $x \in |S|$. Therefore, we identify a suborbifold with its image and regard the structure group of $S$ as a subgroup of the corresponding structure group of $M$. For instance, each stratum in the canonical decomposition \eqref{eq:decom} is a suborbifold (cf. \cite[Proposition 32]{Bor92}).

\begin{defn}[Vector orbibundle]
A rank-$m$ vector orbibundle $(\mathcal V, M, \rho)$ consists of:
\begin{enumerate}
\item $\mathcal V$ and $M$ are smooth orbifolds, and $\rho: \mathcal V \to M$ is a smooth map such that $|\rho|$ is surjective;
\item 
For any $x \in |M|$, there is a local chart $(U_x,\tilde U_x, \Gamma_x, \pi_x)$ around $x$ and a linear action of $\Gamma_x$ on $F$ such that $(\tilde U_x \times V, \Gamma_x)$, with diagonal action, forms a local chart of $\mathcal V$. Moreover, $\rho$ is given by the natural projection $(\tilde U_x \times V)/\Gamma_x \to \tilde U_x/\Gamma_x$.
Note that the fiber of the vector orbibundle over $x$ is the vector orbispace $(V,\Gamma_x)$.
\end{enumerate} 
\end{defn}

For instance, the \emph{tangent bundle} $TM$ of a smooth orbifold $M$ is a vector orbibundle which is locally diffeomorphic to $T \tilde U_x/\Gamma_x$ for local charts $(\tilde U_x, \Gamma_x)$ of $M$. Moreover, each fiber of the tangent bundle over $x$ is the tangent space $T_xM$. The concept of the general $(r,s)$-tensor bundle (denoted by $T^{r,s}M$) and the bundle of $i$-form (denoted by $\Lambda^i M$) can be similarly defined. For the general definition of an orbibundle, see \cite[Section 3.3]{C19}.

Next, we recall the definition of the orbifold fundamental group. First, given a regular point $x \in |M|$, a \emph{special curve} from $x$ is a continuous map $\gamma:[0,1] \to |M|$ such that $\gamma(0)=x$ and $\gamma(t)$ lies in $M_{\mathrm{reg}}$ for all but finitely many $t$. A \emph{special loop} at $x$ is a special loop $\gamma$ with $\gamma(0)=\gamma(1)=x$. Suppose that $(U, \Gamma)$ is a local chart and $\tilde \gamma:[a,b] \to \tilde U$ is a lift of $\gamma_{[a,b]}$, for some $[a,b] \subset [0,1]$. An \emph{elementary homotopy} between two special curves is a homotopy of $\tilde \gamma$ in $\tilde U$, relative to $\tilde \gamma(a)$ and $\tilde \gamma(b)$. A \emph{homotopy} of $\gamma$ is generated by elementary homotopies.

\begin{defn}[Orbifold fundamental group]
Let $M$ be a connected smooth orbifold with a fixed point $x \in M_{\mathrm{reg}}$. The orbifold fundamental group $\pi_1^{\mathrm{orb}}(M, x)$ is the set of homotopy classes of special loops at $x$ with the usual composition law.
\end{defn}

The orbifold fundamental group is independent of the base point $x$ up to isomorphism and is denoted by $\pi_1^{\mathrm{orb}}(M)$. It can be proved \cite[Proposition 13.2.4]{Th02} that any connected smooth orbifold has a universal cover $\rho:\tilde M \to M$, that is, $\rho$ is an orbibundle with discrete fibers. The group of deck transformations of $\rho$ is isomorphic to $\pi_1^{\mathrm{orb}}(M)$. Moreover, if the extra structure is ignored, there exists a natural epimorphism $\pi_1^{\mathrm{orb}}(M) \to \pi_1(|M|)$ (cf. \cite[Proposition 2.5]{BMP03}).

Denote all sections of $\Lambda^i M$ by $\Omega^i$. Then it is clear that there is a well-defined \emph{exterior derivative}
\begin{equation*}
d:\Omega^i \longrightarrow \Omega^{i+1}.
\end{equation*}
Then, we recall the de Rham cohomology of the orbifold.

\begin{defn}[Orbifold de Rham cohomology] 
Let $M$ be a connected smooth orbifold. The orbifold de Rham cohomology $H^*_{\mathrm{dR}}(M)$ is the cohomology groups of the complex
\begin{equation*}
\cdots \overset{d}{\longrightarrow} \Omega^{i-1} \overset{d}{\longrightarrow} \Omega^{i} \overset{d}{\longrightarrow} \Omega^{i+1} \overset{d}{\longrightarrow} \cdots
\end{equation*}
\end{defn}

Similar to the manifold case, we have the following de Rham theorem for orbifolds; see \cite[Theorem 1]{Sa56}.

\begin{thm}[De Rham theorem] \label{thm:der}
Let $M$ be a smooth orbifold. Then $H^*_{\mathrm{dR}}(M) \cong H^*(|M|,\R)$.
\end{thm}

For our applications, we will consider orbifolds with extra structures.

\begin{defn}[Riemannian orbifold] \label{def:ro}
A Riemannian orbifold $(M,g)$ is a smooth orbifold $M$ with a smooth metric $g$, a symmetric $(0,2)$-tensor which is positive everywhere.
\end{defn}

By using the partition of unity (cf. \cite[Lemma 2.11]{KL14}) on the orbifold, it is clear that any smooth orbifold can be equipped with a smooth metric. A Riemannian orbifold $(M,g)$ defines a metric structure described as follows.

We say a curve $\gamma:[0,1] \to |M|$ is \emph{admissible} if $[0,1]$ can be decomposed into a countable number of subintervals $[t_i,t_{i+1}]$ so that $\gamma_{(t_i,t_{i+1})}$ is contained in a single stratum in \eqref{eq:decom}. Each admissible curve $\gamma$ has a well-defined length $L(\gamma)$ (cf. \cite[Theorem 38]{Bor92}). Given two points $x,y \in |M|$, the distance $d(x,y)$ is defined as
\begin{equation*}
d(x,y)=\inf \{L(\gamma) \mid \gamma \text{ is an admissible curve joining $x$ to $y$}\}.
\end{equation*}

Then $(|M|, d)$ becomes a length space. Furthermore, we say that the Riemannian orbifold $(M,g)$ is \emph{complete} if the metric space $(|M|, d)$ is complete. In this case, any two points in $|M|$ can be joined by a minimizing geodesic that realizes the distance $d$. Moreover, it follows from \cite[Theorem 3, P22]{Bor92} that if $\Gamma:[0,1] \to |M|$ is a minimizing geodesic, then for any stratum $\Sigma_{\alpha}$, either $\gamma \subset \Sigma_{\alpha}$ or $\gamma \cap \Sigma_{\alpha} \subset \{\gamma(0)\} \cup \{\gamma(1)\}$. In particular, this implies that $M_{\mathrm{reg}}$ is \emph{geodesically convex} in the sense that any $p,q \in M_{\mathrm{reg}}$ can be connected by a minimizing geodesic contained within $M_{\mathrm{reg}}$.

\begin{defn}[Complex orbifolds]
A \emph{complex} orbifold $(M,J)$ of complex dimension $n$ is a $2n$-dimensional smooth orbifold $M$ such that (i) each local chart $(\tilde U, \Gamma)$ satisfies $\tilde U \subset \C^n$ on which $\Gamma$ acts holomorphically; (ii) An embedding between two local charts $(\tilde U, \Gamma) \to (\tilde U', \Gamma')$ is holomorphic. The complex structure $J$ is an $(1,1)$-tensor induced by the complex structure of local charts.
\end{defn}

\begin{defn}[Symplectic/K\"ahler orbifold]
A \emph{symplectic} orbifold is a smooth orbifold $(M, \omega)$ with a closed non-degenerate $2$-form $\omega$. A \emph{K\"ahler} orbifold $(M,\omega, J)$ is both a symplectic and complex orbifold, where the complex structure $J$ is compatible with the symplectic form $\omega$ in the sense that the bilinear form $g(\cdot, \cdot)=\omega(\cdot, J\cdot)$ defines a Riemannian metric.
\end{defn}

Now, we recall the definition of the group action on orbifolds.
\begin{defn}[Group action] 
Let $G$ be a connected compact Lie group, and let $M$ be an orbifold. A smooth map $\mu: G \times M \rightarrow M$ is called a smooth action of $G$ on $M$ if $|\mu| : G \times |M| \rightarrow |M|$ is a group action.
\end{defn}

For a group $G$-action $\mu$ on $M$, $|\mu|(g,x)$ is simplified as $gx$ for $(g,x) \in G \times |M|$. 
Moreover, for a fixed $x \in |M|$, the isotropy group of $x$ and the orbit through $x$ are denoted by $G_x$ and $G(x)$, respectively. In the following, we will always assume, without loss of generality, that any group action is \emph{effective}, i.e., $\{g \in G \mid gx=x,\, \forall x \in |M|\}=\{1\}$.

Each orbit $G(x)$, as in the manifold case, is a manifold diffeomorphic to $G/G_x$. Moreover, we have the following slice theorem; see \cite[Lemma 2.7, Theorem 2.13]{GKRW17}.

\begin{prop}[Slice theorem] \label{prop:slice1}
For any $x\in |M|$, $G(x)$ is a suborbifold of $M$. Moreover, a $G$-invariant neighborhood of the orbit is equivariantly diffeomorphic to a neighborhood of the zero section in the associated orbibundle $G \times_{G_x} W / \Gamma_x$, where $W =T_{\tilde x} \tilde U_x /T_{\tilde x} \pi_x^{-1}\lc G(x) \rc$ for a local chart $(U_x, \tilde U_x, \Gamma_x, \pi_x)$.
\end{prop}

Here, the action of $G$ induces a representation of $G_x$ on the vector orbispace $(W,\Gamma_x)$. More precisely, the representation is given by a group homomorphism $G_x \to \mathrm{GL}(W/\Gamma_x):=N(\Gamma_x)/\Gamma_x$, where $N(\Gamma_x)$ is the normalizer of $\Gamma_x$ in $\mathrm{GL}(W)$.

Next, we recall

\begin{defn}[Hamiltonian group action]
A $G$-action on a symplectic orbifold $(M,\omega)$ is \emph{symplectic} if any element of $G$ preserves $\omega$. A symplectic $G$-action is \emph{Hamiltonian} if there exists a $G$-equivariant smooth map $\Phi: M \to \g^*$ such that 
\begin{align*}
-\iota(\xi_M) \omega = d \langle\Phi ,\xi\rangle \textit{ for all } \xi \in \g,
\end{align*}
where $\xi_M$ is the fundamental vector field induced by $\xi$. The map $\Phi$ is called the \emph{moment map}. 
\end{defn}

Let $(V, \Gamma, \omega)$ be a symplectic vector orbispace, where $(V,\Gamma)$ is a vector orbispace, and $\omega$ is a symplectic bilinear form on $V$ which is preserved by $\Gamma$. Let $\mathrm{Sp}(V/ \Gamma)$ denote the group of symplectic isomorphisms of $V / \Gamma$, that is, the group $N(\Gamma)/ \Gamma$, where $N(\Gamma)$ is the normalizer of $\Gamma$ in $\mathrm{Sp}(V)$. A \emph{symplectic representation} of a Lie group $H$ on a symplectic vector orbispace $V / \Gamma$ is a group homomorphism $\rho : H \rightarrow \mathrm{Sp}(V / \Gamma)$.
\begin{exmp} \label{ex:mod1}
Let $\rho: H \rightarrow \mathrm{Sp}(V / \Gamma)$ be a symplectic representation of $H$ on a symplectic vector orbispace $(V / \Gamma,\omega)$. Then the action of $H$ on $(V / \Gamma,\omega)$ is Hamiltonian with a moment map $\Phi_{V / \Gamma} : V / \Gamma \rightarrow \h^{*}$, which vanishes at $0$, given by the formula
\begin{align*}
\langle \Phi_{V / \Gamma} (v),\xi \rangle =-\frac{ \omega (\xi_M(v) , v)}{2} \textit{ for all } \xi \in \h \textit{ and } v \in V / \Gamma.
\end{align*}
where $\xi_M(v)$ is the value at $v$ of $\xi_M$. 
\end{exmp}
Any symplectic representation has a natural lift; see \cite[Lemma 3.1]{LT97}.

\begin{lem}\label{lem:lift1}
Let $\rho : H \rightarrow \mathrm{Sp}(V/\Gamma)$ be a symplectic representation of a Lie group $H$ on a symplectic vector orbispace $V / \Gamma$. Then there exists a group $\hat H$, which is an extension of $H$ by $\Gamma$, and a representation $\hat{\rho} : \hat{H} \rightarrow N(\Gamma) \subset \mathrm{Sp}(V)$ so that the following diagram is exact and commutes:
\begin{equation}
\label{D101}
\begin{tikzcd}
1 \arrow[r] & \Gamma \arrow[d, no head, double] \arrow[r] & \hat{H} \arrow[r, "\pi"] \arrow[d, "\hat{\rho}"] & H \arrow[r] \arrow[d, "\rho"] & 1 \\
1 \arrow[r] & \Gamma \arrow[r] & N(\Gamma) \arrow[r] & \mathrm{Sp}(V / \Gamma) \arrow[r] & 1
\end{tikzcd} 
\end{equation}
Moreover, $\hat{\rho}$ is faithfulI if $\rho$ is faithful.
\end{lem} 

\begin{proof}
One can define $\hat H$ as the fiber product of $N(\Gamma)$ and $H$ over $\mathrm{Sp}(V / \Gamma)$.
\end{proof}

If the group is compact and abelian with maximal dimension, we have the following result from \cite[Lemma 6.1]{LT97}.

\begin{lem}\label{lem:lift2}
Let $\rho : H \rightarrow \mathrm{Sp}(V/\Gamma)$ be a faithful symplectic representation of a compact abelian Lie group $H$ on a symplectic vector orbispace $V / \Gamma$ such that $\mathrm{dim}\,V=2 \mathrm{dim}\,H$. Then $H$ and $\hat H$ (from Lemma \ref{lem:lift1}) are torus. 
\end{lem} 

Let $(M, \omega,G, \Phi)$ be a symplectic orbifold with a Hamiltonian $G$-action and moment map $\Phi$. For $x \in |M|$ such that $\Phi(x)$ is fixed by the coadjoint action $\mathrm{Ad}_G^*$, we consider a local chart $(U_x, \tilde U_x, \Gamma_x, \pi_x)$. Then $\omega$ is lifted to a smooth $2$-form (still denoted by $\omega$) on $\tilde U_x$. Next, we define $\lc T_{\tilde x} \pi_x^{-1}(G(x)) \rc^{\omega}$ to be the symplectic complement of $T_{\tilde x} \pi_x^{-1}(G(x))$ and form the symplectic vector orbispace $(V,\Gamma_x, \omega)$, where
\begin{align*}
V = \lc T_{\tilde x} \pi_x^{-1}(G(x)) \rc^{\omega} / T_{\tilde x} \pi_x^{-1}(G(x)).
\end{align*} 
It is clear that by our assumption on $x$, $T_{\tilde x} \pi_x^{-1}(G(x)) \subset \lc T_{\tilde x} \pi_x^{-1}(G(x)) \rc^{\omega}$. Moreover, the induced linear action of $G_x$ on $V/\Gamma_x$ is a symplectic representation of $G$ on $V/\Gamma_x$, which is called the symplectic slice representation.
Now, we can state the symplectic slice theorem from \cite[Lemma 3.5]{LT97}, which was well-known in the manifold case \cite{Ma84, Ma85, GS84}.
\begin{thm}[Symplectic slice theorem]
\label{thm:sslice}
With the above assumptions, for any $G_x$-invariant splitting $\g^*=\g_x^* \oplus \mathrm{ann}(\g_x)$, where $\mathrm{ann}(\g_x)$ denotes the annihilator of $\g_x$ in $\g^*$, there exists a $G$-invariant symplectic form $\omega_Y$ on the orbifold $Y = G \times_{G_x} (\mathrm{ann}(\g_x)\times V / \Gamma_x)$ such that 
\begin{enumerate}
\item a neighborhood of $G(x)$ in M is equivariantly symplectomorphic to a neighborhood of the zero section in Y, and
\item the action of G on $(Y,\omega_Y)$ is Hamiltonian with a moment map $\Phi_Y:Y \rightarrow \g^{*}$ given by
\begin{align*}
\Phi_Y ( [g,\eta,v]) = \Phi(x)+\mathrm{Ad}^*_g \lc -\eta + \Phi_{V/ \Gamma_x}(v) \rc,
\end{align*}
where $\Phi_{V / \Gamma_x} : V / \Gamma_x \rightarrow \g_x^{*}$ is the moment map for the slice representation of $G_x$ on $V/\Gamma_x$ as in Example \ref{ex:mod1}. 
\end{enumerate}
\end{thm}
In fact, $(Y, \omega_Y)$ is obtained via symplectic reduction $(T^* G \times V /\Gamma_x) \sslash G_x$, where the action of $G_x$ on $T^*G=G \times \g$ is on the right. Specifically, the zero level set in $T^* G \times V/\Gamma_x$ of the moment map consists of $(g, -\Phi_{V/\Gamma_x}(v)+h,v)$, where $g \in G$, $h \in \mathrm{ann}(\mathfrak g_x)$ and $v \in V/\Gamma_x$. Consequently, the level set is isomorphic to $G \times \mathrm{ann}(\mathfrak g_x) \times V/\Gamma_x$. Thus, Proposition \ref{thm:sslice} follows from Proposition \ref{prop:slice1} and the $G$-relative Darboux theorem \cite[Theorem 6]{BL97} with the necessary adaptations for the orbifold setting. It is important to note that the identification in the above result depends on the $G_x$-equivariant splitting of $\g^*$.
In our application, we require the following concept, where the group is a maximal torus.
\begin{defn}[Symplectic toric orbifold] \label{def:sto}
A \emph{symplectic toric orbifold} $(M^{2n}, \omega, \T^n, \Phi)$ is a symplectic orbifold $(M^{2n}, \omega)$ with a Hamiltonian $\T^n$-action and moment map $\Phi$.
\end{defn}
We denote the Lie algebra of $\T^n$ by $\lt$. For a symplectic toric orbifold $(M^{2n}, \omega, \T^n, \Phi)$ and any $x \in |M|$, we set $H:=\T_x$ with Lie algebra $\h$ and lattice $\mathfrak i$. Since the symplectic slice representation of $H$ on $V/\Gamma_x$ is faithful and $2\mathrm{dim}\,\hat \h=\mathrm{dim}\, V$, it follows from Lemma \ref{lem:lift2} that both $H$ and its lift $\hat H$ are tori. Next, we choose a subtorus $K$ with Lie algebra $\mathfrak k$ so that $\lt=\mathfrak k \times \h$. Therefore, we can identify $\mathrm{ann}(\h)=\mathfrak k^*$ and hence $\lt^*=\mathfrak k^* \times \h^*$. Consequently, we have the identification 
\begin{align*}
\T^n \times_{H} (\mathfrak k^* \times V / \Gamma_x)=T^* K \times V / \Gamma_x,
\end{align*}
and the following result is immediate from the symplectic slice theorem \ref{thm:sslice}.
\begin{cor}\label{cor:slice1}
Let $(M^{2n}, \omega, \T^n, \Phi)$ be a symplectic toric orbifold. Then, for any $x \in |M|$, there exists a $\T$-invariant neighborhood $U$ of $\T(x)$ such that $U$ is equivariantly symplectomorphic to a neighborhood of the $K \times \{0\}$ in
\begin{align*}
T^* K \times V / \Gamma_x,
\end{align*}
equipped with the standard symplectic form, where $K$ is a subtorus with $\T^n=\T_x \times K$. Moreover, the moment map restricted to $U$ is
\begin{align*}
\Phi ((k,\eta,v)) = \Phi(x) - \eta + \Phi_{V/\Gamma_x}(v),
\end{align*}
where $(k,\eta) \in T^*K=K \times \mathfrak k^*$ and $v \in V/\Gamma_x$.
\end{cor}
If a symplectic toric orbifold $(M^{2n}, \omega, \T^n, \Phi)$ has a $\T^n$-invariant compatible complex structure $J$. it is not always possible for the $\T^n$-action to extend to a holomorphic $(\C^*)^n$-action; see \cite[Example 1]{Ci22}. The key issue is that for some $\xi \in \lt$, the vector field $J \xi_M$ may not be complete. Note that this problem does not arise if $M$ is compact. When the extension is possible, we define the following:
\begin{defn}[K\"ahler toric orbifold] \label{def:kto}
A \emph{K\"ahler toric orbifold} $(M^{2n}, \omega, J, (\C^*)^n, \Phi)$ is a K\"ahler orbifold $(M^{2n}, \omega, J)$ with a holomorphic action of $(\C^*)^n$ such that its restriction to $\T^n$ is a Hamiltonian action with moment map $\Phi$.
\end{defn}
For a K\"ahler toric orbifold $(M^{2n}, \omega, J, (\C^*)^n, \Phi)$ and any $x \in |M|$, we set as above $H=\T_x$ with complexification $H_{\C}$. Additionally, we choose a complementary subtorus $K$ with complexification $K_{\C}$. For later application, we prove the following holomorphic slice theorem, which essentially follows from \cite[Theorem 1.12, Theorem 1.23]{Sj95}. We will sketch and modify the proof to accommodate the orbifold setting.
\begin{thm}[Holomorphic slice theorem]
\label{thm:hslice}
Let $(M^{2n}, \omega, J, (\C^*)^n, \Phi)$ be a K\"ahler toric orbifold. Then, for any $x \in |M|$, there exists a complex vector orbispace $(V, \Gamma_x)$ with a faithful complex representation of $H^{\C}$ such that there exists a $(\C^*)^n$-equivariant biholomorphic map $\psi: K_{\C} \times V/\Gamma_x$ onto a $(\C^*)^n$-invariant open neighborhood of $x$.
\end{thm}

\begin{proof}
We set $W=\tilde{T}_x M$, the uniformized tangent space at $x$. It is clear that $(\C^*)^n_x=H_{\C}$ and $H_{\C}$ induces a complete representation $\rho: H_{\C} \to \mathrm{GL}(W/\Gamma_x,\C)$, which restricts to $H$ as a unitary representation $H \to \mathrm{U}(W/\Gamma_x)$. It follows from Lemma \ref{lem:lift1} that $H$ lifts to a torus $\hat H \to \mathrm{U}(W)$ with complexification $\hat \rho: \hat H_{\C} \to \mathrm{GL}(W,\C)$.

The lift of the tangent space to the complex orbit $(\C^*)^n(x)$ forms a complex subspace of $W$, and we denote its orthogonal complement by $V$. It is clear that $(V, \Gamma_x)$ is a complex vector orbispace and the representation $\rho$ restricts to $H_{\C}$ as $\rho: H_{\C} \to \mathrm{GL}(V/\Gamma_x,\C)$. In particular, $V$ is $\hat H_{\C}$-invariant. 

Let $(\tilde U, \Gamma_x)$ be a holomorphic local chart around $x$, where $\tilde U \subset \C^n$. For the orbifold $U=\tilde U/\Gamma_x$, we choose a biholomorphic map $\phi: U \to M$ with $\phi(0)=x$ and $d\phi_0=\mathrm{id}_{\C^n}$. By shrinking $\tilde U$ if possible, we assume $\tilde U$ is $\hat H$-invariant. Then we set $U'=\phi(U)$ and let $\psi: U' \to U$ be the inverse of $\phi$. After averaging over $H$ and further shrinking $U'$ if necessary, we may assume that both $\psi$ and $\phi$ are $H$-invariant.

Next, we set $B':=V \cap \tilde U$, $B=B'/\Gamma_x$ and $S=H_{\C}(B)$. Then, one can prove that if $\tilde U$ is sufficiently small, the map $\phi: B \to M$ can be uniquely extended to a $(\C^*)^n$-equivariant map from $(\C^*)^n \times_{H_{\C}} S$ onto an open set of $M$ (see \cite[Theorem 1.12]{Sj95} for details).

Since the action of $H_{\C}$ on $V$ is effective with $\mathrm{dim}_{\C}\, V=\mathrm{dim}_{\C}\, H_{\C}$, it is clear that $S=W$. Additionally, since $(\C^*)^n/H_{\C}=K_{\C}$, the conclusion follows.
\end{proof}

\subsection{Symplectic toric orbifolds and labeled polyhedra}
In this section, we consider a symplectic toric orbifold $(M^{2n}, \omega,\T^n, \Phi)$ (cf. Definition \eqref{def:sto}). In the case of manifolds, Delzant \cite{De88} proved that compact toric manifolds correspond one-to-one with smooth and simple polytopes through the moment map (cf. Definition \ref{def:poly}). Later, this result was generalized by Lerman and Tolman \cite{LT97} to compact toric orbifolds, establishing a connection with rational and simple polytopes. We will prove a similar correspondence for a class of symplectic toric orbifolds.
First, we recall the following definitions. Throughout, we assume $\lt$ to be the Lie algebra of $\T^n$ with lattice $\mathfrak l$, $\lt^*$ to be its dual, and $\mathfrak l^*$ the dual lattice.
\begin{defn} \label{def:poly}
Let $\lt$ be the Lie algebra of $\T^n$ with lattice $\mathfrak l$, and $\lt^*$ be its dual and $\mathfrak l^*$ dual lattice.
\begin{enumerate}
\item \emph{(Polyhedron)} A \emph{polyhedron} is a finite intersection of affine half spaces $H^+_{n_i, a_i} := \{ x \in \mathfrak{t}^* \; | \; \langle n_i , x \rangle +a_i \ge 0\}$ with $n_i \in \lt, a_i \in \R$. A \emph{face} of a polyhedron $P$ is a polyhedron $P\cap H_{n, a}$, where $H_{n,a}:= \{ x \in \mathfrak{t}^* \; | \; \langle n , x \rangle +a =0\}$ is an affine hyperplane such that $P \subset H^+_{n,a}$. Faces of dimension $n-1$, $1$ and $0$ are called \emph{facet}, \emph{edge} and \emph{vertex}, respectively.

\item \emph{(Polytope)} A \emph{polytope} is a bounded polyhedron. Equivalently, a polytope is the convex hull of finite vertices.
\item \emph{(Polyhedral cone)} A \emph{polyhedral cone} is a finite intersection of $H^+_{n_i, 0}$. Equivalently, a polyhedral cone is the convex hull of finite rays from $0$. A polyhedral cone is called \emph{rational} if each edge is generated by an element in $\mathfrak l$. A polyhedral cone is called \emph{proper} if it does not contain a line.

\item \emph{(Asymptotic cone)} For a polyhedron $P=\cap H^+_{n_i,a_i}$, its asymptotic cone, denoted by $C(P)$, is a polyhedral cone $C(P):=\cap H^+_{n_i,0}$.
\item \emph{(Proper/Rational/Simple)} A polyhedron $P$ is \emph{proper} if its asymptotic cone $C(P)$ is proper. $P$ is \emph{rational} if, for each vertex $x \in P$, the polyhedral cone $C_v:=\mathrm{Cone}(\R_+(P-v))$ is rational. $P$ is called \emph{simple} if there are exactly $n$ edges at any vertex, and the generators of these edges form a basis of $\lt^*$ over $\R$.
\end{enumerate}
\end{defn}

For a polyhedron, we have the following structure result (cf. \cite[Theorem 1.2]{Z07}).

\begin{lem}\label{lem:mink}
Any polyhedron $P$ can be decomposed as
\begin{align*}
P=\mathrm{Conv}(\{v_i\})+C(P),
\end{align*}
where $\{v_i\}$ is the set of vertices of $P$, and sum is the Minkowski sum.
\end{lem}



For a proper, rational, simple polyhedron $P$, being simple implies that $P$ must be of dimension $n$. Therefore, it has a representation
\begin{align} \label{eq:rep}
P=\bigcap_{i=1}^N H^+_{n_i, a_i}, 
\end{align}
so that each $P \cap H_{n_i, a_i}$ is a facet. Moreover, being rational means that each $n_i$ can be chosen to be a primitive element of the lattice $\mathfrak l$; that is, there is no integer $k1$ such that $n_i/k \in \mathfrak l$. Consequently, the representation \eqref{eq:rep} is unique. Being proper indicates that $P$ itself does not contain an affine line. Moreover, the dual polyhedral cone $C'(P):=\{v \in \lt \mid \la v, w \ra \ge 0,\, \forall w \in C(P)\}$ has a nonempty interior so that there exists $b \in \lt$ such that the function $\la b,\cdot \ra$ on $P$ is proper and bounded below (cf. \cite[Appendix]{PW94}).

Now, we have the following definition.
\begin{defn} \label{def:orb1}
A symplectic toric orbifold $(M^{2n}, \omega,\T^n, \Phi)$ satisfies $(\dagger)$ if
\begin{enumerate}[label=(\alph*)]
\item $\T^n$ has finite fixed points.

\item There exists $b \in \lt$ such that the function $\la \Phi, b \ra$ on $M$ is proper and bounded below. In particular, it implies that $\Phi$ itself is proper.
\end{enumerate}
\end{defn}

Note that any compact symplectic toric orbifold satisfies $(\dagger)$. In this case, Lerman and Tolman proved the following result \cite[Theorem 5.1, Theorem 5.2]{LT97}, which is a generalization of the celebrated Atiyah-Guillemin-Sternberg convexity theorem for Hamiltonian torus actions on compact symplectic toric manifolds \cite{AT82}\cite{GS82}.

\begin{thm}
\label{thm:convex1}
Let $(M^{2n}, \omega,\T^n, \Phi)$ be a compact symplectic toric orbifold. Then, the image of $\Phi$ is a rational and simple polytope, which is the convex hull of the image of points in $M$ fixed by $\T^n$. Moreover, each fiber $\Phi^{-1}(a)$ is connected.
\end{thm}

For a general symplectic toric orbifold, even in the manifold cases, it is possible that the image of the moment map is not convex (cf. \cite[Counterexample 0.1]{Pr94}). Next, we prove the following result, which is a generalization of Theorem \ref{thm:convex1}.

\begin{thm}
\label{thm:image}
Let $(M^{2n}, \omega,\T^n, \Phi)$ be a symplectic toric orbifold with $(\dagger)$. Then the following assertions hold:
\begin{enumerate}[label=(\roman*)]
\item $\Phi(M)$ is a proper, rational, simple polyhedron.
\item $\Phi: M \rightarrow \Phi(M)$ is an open mapping.
\item The inverse images of points in $\Phi(M)$ are connected.
\item $x \in M$ is fixed by $\T^n$ iff $\Phi(x)$ is a vertex of $\Phi(M)$.
\end{enumerate}
\end{thm}

\begin{proof}
We set $P$ to be the image of $\Phi$. For any $x \in |M|$, let $H:=\T_x$ with Lie algebra $\h$ and lattice $\mathfrak i$. We then choose a subtorus $K$ with Lie algebra $\mathfrak k$ so that $\lt=\mathfrak k \times \h$. It follows from Corollary \ref{cor:slice1} that there exists a $\T$-invariant neighborhood $U$ of $\T(x)$ such that $U$ is equivariantly symplectomorphic to a neighborhood of $K \times \{0\}$ in
\begin{align}
\label{E201aa}
T^* K \times V / \Gamma_x,
\end{align}
equipped with the standard symplectic form. Moreover, the moment map restricted to $U$ is
\begin{align}
\label{E201a}
\Phi ((k,\eta,v)) = \Phi(x) - \eta + \Phi_{V/\Gamma_x}(v),
\end{align}
where $(k,\eta) \in T^*K=K \times \mathfrak k^*$ and $v \in V/\Gamma_x$.

We claim that the cone $C = \Phi_{V/ \Gamma_x} (V / \Gamma_x)$ is a rational, simple polyhedral cone in $\h^*$. Indeed, let $\pi : \hat{H} \rightarrow H$ be the group obtained from Lemma \ref{lem:lift1}, then $\hat H$ is also a torus from Lemma \ref{lem:lift2}. It follows from a standard result (cf. \cite[Lemma 3.3]{LT97}) that $(V,\omega)$ admits a decomposition $(V,\omega)=\oplus_{i=1}^m( V_i, \omega_i)$ by invariant mutually perpendicular $2$-dimensional symplectic subspaces. In addition, the representation $\hat \rho$ of $\hat H$ on $(V_i, \omega_i)$ has weight $\beta_i$. Moreover, the moment map $\Phi_{V}: V \to \hat \h^*$ is given by
\begin{align}
\label{E201}
\Phi_{V} (v_1,\cdots,v_m) = \sum |v_i|^2 \beta_i \; \textit{ for all } v = (v_1,\cdots,v_m) \in \oplus_i V_i,
\end{align}
where $m=\mathrm{dim}\,\h$ and $|\cdot|$ is an invariant norm compatible with $\omega$. Since the representation $\hat{\rho}$ is faithful, the vectors $\{\beta_i\}$ for a $\Z$-basis of $\hat{\mathfrak i}^*$. Let $\{e_i\}$ denote the dual basis of the lattice $\hat{\mathfrak i}$. It is clear that the diagram
\begin{equation} \label{eq:dia1}
\begin{tikzcd}
V \arrow[d, "\hat{\Phi}_V"] \arrow[r] & V / \Gamma_x \arrow[d, "\Phi_{V / \Gamma_x}"] \\
\hat{\mathfrak{h}}^* & \h^* \arrow[l, "\varpi^*"] 
\end{tikzcd}
\end{equation}
commutes, where $\varpi : \hat{\h} \rightarrow \h$ is the isomorphism of Lie algebras induced by $\pi: \hat{H} \rightarrow H$ and $\varpi^*$ is its transpose. Therefore, the image of $\Phi_{V / \Gamma_x}$ is the rational simple polyhedral cone:
\begin{align*}
\Phi_{V / \Gamma_x} (V / \Gamma_x) = \{\eta \in \h^*\;|\; \langle \varpi{(e_i)},\eta \rangle \geq 0 \},
\end{align*}
which finishes the proof of the claim. 
Next, we define the polyhedral cone $C_x:=\mathfrak k^* \times C \subset \lt^*$. It is clear from \eqref{E201a}, \eqref{E201} and \eqref{eq:dia1} that for any open neighborhood $V$ of $x$, there exists a smaller neighborhood $x\in U_x \subset V$ such that $\Phi(U_x)$ is an open subset of $P$, which takes the form of an open neighborhood of $\Phi(x)$ in $\Phi(x)+C_x$. Moreover, $\Phi^{-1}(\Phi(y)) \cap U_x$ is connected for any $y \in U_x$. Since Definition \ref{def:orb1} (b) implies that $\Phi$ is proper, it follows from the ``local-global-principle'' \cite[Theorem 3.19]{HNP94} that $\Phi: M \to P$ is an open such that $\Phi^{-1}(a)$ is connected for any $a \in P$. In addition, $P$ is a closed, convex, locally polyhedral set in the sense that for any $a \in P$, there exists a neighborhood $V$ such that $V \cap P=V \cap (a+\overline{\R^+(P-a)})$. 
Notice that $\overline{\R^+(P-a)}=\R^+(P-a)=C_x$ for $x \in \Phi^{-1}(a)$ since $P$ is closed and convex. In particular, if $a=\Phi(x)$ is an extreme point of $P$, we must have $\mathfrak k^*=0$ and hence $H=\T^n$. Conversely, if $x$ is a fixed point of $\T^n$, $\Phi(x)$ is an extreme point by \eqref{E201}. Therefore, the assertions (ii), (iii), and (iv) of the theorem are proved.
Note that $P$ has at least one extreme point; otherwise, $P$ must contain a line because $P$ is closed and convex, which would contradict the assumption in Definition \ref{def:orb1} (b). Moreover, it follows from Definition \ref{def:orb1} (a) that $P$ has finitely many extreme points, denoted by $\{a_1,\cdots, a_N\}$.

Therefore, to prove (i), we only need to show that $P$ is a polyhedron. This follows immediately from the identity
\begin{align} \label{E201b}
P=\bigcap_{i=1}^N L_{a_i},
\end{align}
where $L_{a_i}:=\R^+(P-a_i)$ is a polyhedral cone. It is clear that $P \subset \bigcap_{i=1}^N L_{a_i}$. To prove that converse, we equip $\lt^*$ with an inner product that identifies $\lt$ with $\lt^*$. Suppose, to the contrary, that there exists $ y\in \bigcap_{i=1}^N L_{a_i} \backslash P$. Let $y' \in P$ be the nearest point in $P$ to $y$. Then, the affine hyperplane $H_{v,b}= \{ x \in \mathfrak{t}^* \; | \; \langle v , x \rangle +b =0\}$ for $v=y'-y$ and $b=-\la y',v \ra$ is a supporting hyperplane of $P$. Since $P$ contains no affine line, $P \cap H_{n,b}$ has an extreme point $a_i$. Since $y \in L_{a_i}$, according to the definition of $L_{a_i}$, there exists $z \in P$ located on the interior of the segment formed by $y$ and $a_i$. However, $z \notin H_{v,b}^+$, leading to a contradiction.
In conclusion, the identity \eqref{E201b} holds, and the proof of the theorem is complete. 
\end{proof}

As in the case of compact symplectic orbifolds, the polyhedron itself cannot fully determine the symplectic orbifold. Therefore, we need the concept of a \emph{labeled polyhedron}, introduced in \cite{LT97}, to capture information about the structure groups.

\begin{defn}
A \emph{labeled polyhedron} $P$ in $\lt^*$ is a proper, rational, simple polyhedron plus a positive integer attached to each facet. Two labeled polyhedra are \emph{isomorphic} if one can be mapped to the other by a translation, and the corresponding open facets have the same integer labels.
\end{defn}
Note that any labeled polyhedron $P$ can be expressed uniquely in the following way: 
\begin{align} \label{eq:rep2}
P = \bigcap_{i = 1}^{N} \{ x \in \lt^* \; | \langle x, m_i n_i \rangle+a_i \geq 0\},
\end{align} 
where $n_i \in \mathfrak l$ is primitive, and $F_i:=P \cap \{ x \in \lt^* \; | \langle x, m_i n_i \rangle+a_i =0\}$ is the $i$-th facet, to which the integer $m_i$ is attached.
Given a symplectic toric orbifold $(M^{2n}, \omega,\T^n, \Phi)$ with $(\dagger)$, we define an associated labeled polyhedron. First, we set $P=\Phi(M)$ and assign an integer to each facet of $P$ following \cite{LT97}. Let
\begin{align*}
P=\bigcap_{i=1}^N H^+_{n_i, a_i},
\end{align*}
so that $n_i \in \mathfrak l$ is primitive and the $i$-th facet is $F_i:=P \cap H_{n_i,a_i}$. For any $\Phi(x) \in \mathring{F_i}$, by using the same notations as in the proof of Theorem \ref{thm:image}, we know that both $\hat H$ and $H$ are circles. For $\{e_1\}$ being the $\Z$-basis of $\hat{\mathfrak{i}}$, we define $m_i>0$ to be the integer so that $ \varpi (e_1)=m_i n_i$. It follows from the diagram \eqref{D101} that $\Gamma_x=\Z/m_i \Z$. Moreover, it follows from \eqref{E201aa}, \eqref{E201a} and the fact that $\mathring{F_i}$ is connected that any point in $\Phi^{-1}(\mathring{F_i})$ has the same structure group $\Z/m_i \Z$. Therefore, $m_i$ is well-defined for the facet $F_i$.

Generally, one can read the structure group of any $x \in M$ from the labeled polyhedron $P$. Indeed, we assume, without loss of generality, that $\Phi(x)$ lies in the interior of $\cap_{i=1}^m F_i$. Then the isotropy group $H$ of $x$ is the subtorus of $\T^n$ so that its Lie algebra $\h$ is the linear span of $\{n_i\}_{1 \le i \le m}$. Moreover, the structure group $\Gamma_x$ at $x$ is isomorphic to $\mathfrak i/\mathfrak i'$, where $\mathfrak i$ is the lattice of $\h$, and $\mathfrak i'$ is the subgroup generated by $\{m_i n_i\}_{1 \le i \le m}$; see \cite[Lemma 6.6]{LT97}.

Next, we show that a labeled proper, rational, simple polyhedron determines the symplectomorphism of the symplectic toric orbifold. 
\begin{thm}
\label{T204}
If two symplectic toric orbifolds $(M^{2n}, \omega,\T^n, \Phi)$ and $(M'^{2n}, \omega',\T^n, \Phi')$ with $(\dagger)$ have isomorphic labeled polyhedra, then they are isomorphic as symplectic toric orbifolds.
\end{thm}
\begin{proof}
The result follows essentially from \cite[Theorem 7.4]{LT97}, and we sketch the proof for readers' convenience.
We set $P=\Phi(M)=\Phi'(M')$ to be the common labeled polyhedron. Then it follows from \cite[Proposition 6.5]{LT97} that $(M^{2n}, \omega,\T^n, \Phi)$ and $(M'^{2n}, \omega',\T^n, \Phi')$ are locally isomorphic over $P$ in the sense that every point $a$ in P has a neighborhood $U_a$ such that $\Phi'^{-1}(U_a)$ and $\Phi^{-1}(U_a)$ are isomorphic as symplectic toric orbifolds. Then we define a sheaf $ \mathcal{S} $ over $P$ such that for every open set $U \subset P$, $\mathcal{S}(U)$ is the set of isomorphisms of $\Phi^{-1}(U)$. Suppose that $\mathcal{U} = \{U_i\}$ be a countable covering of $P$ such that there is an isomorphism $h_i : \Phi^{-1}(U_i) \rightarrow \Phi'^{-1}(U_i)$ for each $i$. We define $f_{ij} : \Phi^{-1}(U_i \cap U_j) \rightarrow \Phi^{-1}(U_i \cap U_j)$ by $f_{ij} = h_i^{-1} \circ h_j$. Then $\{f_{ij}\} \in C^1(\mathcal{U},\mathcal{S})$ and it is closed, i.e., $\{f_{ij}\} \in Z^1 (\mathcal{U}, \mathcal{S})$. Therefore, the global isomorphism between $(M^{2n}, \omega,\T^n, \Phi)$ and $(M'^{2n}, \omega',\T^n, \Phi')$ follows from the fact that the \v{C}ech cohomology $H^1(P, \mathcal S)=0$.
Let $\underline{\ell \times \R}$ denote the sheaf of locally constant functions with values in $\ell \times \R$, where $\ell \subset \lt$ is the lattice group. Next, we define a sheaf $\mathcal{C}^{\infty}$ over $P$ as follows: for each open set $U \subset P$, $\mathcal{C}^{\infty}(U)$ is the set of $\T^n$-invariant smooth functions on $\Phi^{-1}(U)$. It is clear that $H^i (P,\underline{\ell \times \R})=H^i (P,\mathcal{C}^{\infty})=0$ for any $i>0$ since $P$ is contractible and the fact that $\mathcal{C}^{\infty}$ is a fine sheaf.
Now, we have the short exact sequence of sheaves:
\begin{equation} \label{eq:ses}
\begin{tikzcd}
0 \arrow[r] & \underline{\ell\times \R} \arrow[r, "j"] & \mathcal{C}^{\infty} \arrow[r, "\Lambda"] & \mathcal{S} \arrow[r] & 0.
\end{tikzcd}
\end{equation}
Here, $j$ is defined by $j ( \xi ,c) (x) = c + \langle \xi, \Phi(x) \rangle$ for $(\xi,c) \in \ell \times \R$. For the map $\Lambda$, let $U \subset P$ be an open set and let $f : \Phi^{-1}(U) \rightarrow \R$ be a smooth $\T^n$-invariant function. Then we define the vector field $X_f$ satisfying $\iota_{X_f} \omega = df$, and define $\Lambda(f)$ to be the time one flow of $X_f$.
Therefore, the short exact sequence \eqref{eq:ses} yields immediately that $H^i(P,\mathcal{S}) = 0$ for $i > 0$.
\end{proof}
Conversely, for each labeled polyhedron, one can construct the corresponding symplectic toric orbifold with $(\dagger)$. The construction is essentially Delzant's construction \cite{De88} by symplectic reduction with minor modifications as \cite[Theorem 8.1]{LT97}.

\begin{thm}\label{T206}
Given a labeled polyhedron $P \subset \lt^*$, there exists a symplectic toric orbifold $(M^{2n},\omega,\T^n,\Phi)$ with $(\dagger)$ such that $\Phi(M) = P$ and the orbifold structure group of a point in $M$ which maps to the interior of the facet $F$ is $\Z / m \Z$, where $m$ is the integer attached to $F$.
\end{thm}
\begin{proof}
We sketch the construction. We write $P$ as in \eqref{eq:rep2}
\begin{align*}
P = \bigcap_{i = 1}^{N} \{ x \in \lt^* \; | \langle x, m_i n_i \rangle+a_i \geq 0\}.
\end{align*} 
We define a linear projection $\varpi: \R^N \rightarrow \lt$ by $\varpi(e_i) = m_i n_i$, where $\{e_i\}$ is the standard basis for $\R^N$. The kernel of $\varpi$, denoted by $\mathfrak k$, is the Lie algebra of a subtorus $K \subset \T^N$. Therefore, we obtain a short exact sequence and its dual : 
\begin{equation*}
\begin{tikzcd}
0 \arrow[r] & \mathfrak{k} \arrow[r, "i"] & \R^N \arrow[r, "\varpi"] & \lt \arrow[r] & 0, \\
0 \arrow[r] & \lt^* \arrow[r, "\varpi^{*}"] & (\R^N)^* \arrow[r, "i^*"] & \mathfrak{k}^* \arrow[r] & 0,
\end{tikzcd}
\end{equation*} 
where $i$ is the inclusion. Equiped $\C^N$ with standard symplectic form $\sum_j \frac{\sqrt{-1}}{2} d z_j \wedge d \bar{z}_j$. The standard $\T^N$-action on $\C^N$ has moment map 
\begin{align*}
\Phi_{\T^N} (z_1,\cdots,z_N) = \frac{1}{2}\sum_{i =1}^N |z_i|^2 e_i^*,
\end{align*}
where $\{e_i^*\}$ is the basis dual to $\{e_i\}$. Next, we define an affine embedding $l: \lt^* \rightarrow (\R^N)^*$ by $l(x) = \varpi^*(a)+ a$, where $x=\sum_i a_i e^*_i$. Note that 
\begin{align*}
l (P) 
& = \{ \xi \in (\R^N)^* \; | \; \xi \in l(\lt^*) \textrm{ and } \xi_i \geq 0 \textrm{ for all } i\}= \{ \xi \in (\R^N)^* \; | \; \xi \in (i^*)^{-1} (i^*(a)) \textrm{ and } \xi_i \geq 0 \textrm{ for all } i\}.
\end{align*}
The Hamiltonian $K$-action on $\C^N$ has the moment map $\Phi_K = i^* \circ \Phi_{\T^N}$. It is not hard to prove that $i^*(a)$ is a regular value of $\Phi_K$, and the reduced space $M = \Phi^{-1}_K (i^*(a)) / K$ is a symplectic toric orbifold. The action of $\T^N$ descends to a Hamiltonian action of $\T^n=\T^N/K$, and the image of the corresponding moment map is exactly $P$. Moreover, one can show that for $[z] \in M$ mapping to the interior of the $i$-th facet of $P$, the structure group at $[z]$ is exactly the isotropy group $K_z$ for $z \in \C^N$, which is easy to see equal to $\Z /m_i \Z$. 
It is clear that the constructed symplectic toric orbifold satisfies $(\dagger)$ since $P$ satisfies $(\dagger)$. In sum, the proof is complete.
\end{proof}

For a given labeled polyhedron $P$, we denote the associated symplectic toric orbifold obtained in Theorem \ref{T206} by $(M_P, \omega_P)$, which carries a compatible $\T^n$-invariant complex structure $J_P$ inherited from that of $\C^N$. In addition, the torus $\T^n$-action can be extended to a holomorphic $(\C^*)^n$-action on $M_P$, and hence $(M_P,J_P, (\C^*)^n)$ is a toric variety, whose concept we recall below; see \cite[Theorem VII.2.1]{Au04} or \cite[Lemma 2.1]{BGL08}.

\begin{defn}[Toric variety] \label{def:tr}
A toric variety $X$ is an irreducible normal variety containing a complex torus $(\C^*)^n$ as a Zariski open subset such that the action of $(\C^*)^n$ on itself extends to an algebraic action of $(\C^*)^n$ on X.
\end{defn}

It is known (cf. \cite[Corollary 3.18]{CLS11}) that a toric variety $X_{\Sigma}$ can be obtained from a fan $\Sigma \subset \lt^*$ by first constructing an affine toric variety for each cone in $\Sigma$ and then gluing all pieces together. In particular, $(M_P,J_P, (\C^*)^n)$ as a toric variety is determined by the normal fan of the polyhedron $P$.

Combined with Theorem \ref{T204}, the following result is immediate; see also \cite[Theorem 9.1]{LT97}.

\begin{thm}\label{thm:ka}
Any symplectic toric orbifold $(M^{2n}, \omega,\T^n, \Phi)$ with $(\dagger)$ has a complex structure $J$ so that it becomes a K\"ahler toric orbifold. Moreover, $(M,J, (\C^*)^n)$ is a toric variety with a fan equal to the normal fan of $\Phi(M)$.
\end{thm}

For a general K\"ahler toric orbifold (cf. Definition \eqref{def:kto}), we prove the following result similar to \cite[Lemma 2.14]{Ci22}; see also \cite[Lemma 9.2]{LT97} and \cite[Proposition 1]{Ab03}.

\begin{thm}\label{thm:to}
Let $(M^{2n}, \omega,J, (\C^*)^n, \Phi)$ be a K\"ahler toric orbifold with $(\dagger)$. Then, there exists a $(\C^*)^n$-equivariant biholomorphism $\varphi:(M_P,J_P) \to (M,J)$, where $P=\Phi(M)$. Moreover, $[\varphi^* \omega]=[\omega_P]$ in the orbifold de Rham group $H^2_{\mathrm{dR}}(M_P)$.
\end{thm}

\begin{proof}
It follows from Theorem \ref{T204} that there exists a $\T^n$-equivariant symplectomorphism $\psi: (M_P, \omega_P, \T^n) \to (M, \omega, \T^n) $. We set $J':=\psi^* J$, which, by our assumption, is a compatible $\T^n$-invariant complex structure on $(M_P, \omega_P, \T^n)$ such that the action extends to $(\C^*)^n$.

We claim that $(M_P, J_P,(\C^*)^n )$ is equivariantly biholomorphic to $(M_P, J',(\C^*)^n )$ via a map $\phi$. In addition, the biholomorphism $\phi$ acts as the identity on the cohomology group $H^2(|M_P|, \Z)$. Indeed, the claim can be proved similar to \cite[Proposition 1]{Ab03}; see also \cite[Lemma 9.2]{LT97}. We set $x_1,\cdots,x_m$ are all points fixed by $(\C^*)^n$. For each $x_i$, it follows from Theorem \ref{thm:hslice} that there exists a $(\C^*)^n$-invariant neighborhood $U_i$, which is $(\C^*)^n$-equivariantly biholomorphic to $\C^n/\Gamma_i$, where $\Gamma_i$ is the structure group at $x_i$. In particular, the neighborhood $U_i$ is an affine toric variety on which $(\C^*)^n$-action is linearized. It is not difficult to show that the collection $\{U_i\}$ covers $M_P$, and all the transition maps are rational and entirely determined by the combinatorics of $P$. Therefore, we can define a biholomorphism $\phi: (M_P, J_P) \to (M_P, J')$; see, for example, \cite[Lemma 2.14]{Ci22}.

We set the $\T^n$-invariant Weil divisor $D_i:=\Phi_P^{-1}(F_i)$, where $F_i$ is the $i$-th facet of $P$. Note that the map $\phi$ constructed preserves all $D_i$. Therefore, $\phi$ acts as the identity on the Picard group $\mathrm{Pic}(M_P)$, since $\mathrm{Pic}(M_P)$ is generated by all $D_i$ (cf. \cite[Theorem 4.2.1]{CLS11}). Consequently, it follows from \cite[Theorem 12.3.2]{CLS11} that $\phi$ acts as the identity on $H^2(|M_P|, \Z)$.

Finally, the equivariant biholomorphism is defined as $\varphi=\phi \circ \psi$, which satisfies $[\varphi^* \omega]=[\omega_P]$ in $H^2(|M_P|,\Z)$ and hence in $H^2_{\mathrm{dR}}(M_P)$ by Proposition \ref{thm:der}.
\end{proof}

\section{K\"ahler geometry of toric orbifolds}

\subsection{K\"ahler metrics on toric orbifolds}

In this section, we consider a K\"ahler toric orbifold $(M^{2n}, \omega, J, (\C^*)^n, \Phi)$ with ($\dagger$) (cf. Definition \ref{def:kto},\ref{def:orb1}), and recall the Abreu-Guillemin theory on the K\"ahler metrics (cf. \cite{G94}\cite{Ab03}\cite{Ab01}\cite{Ap19}).

Let $P=\Phi(M)$ be the labeled polyhedron. By \eqref{eq:rep2}, $P$ is uniquely represented as: 
\begin{align} \label{eq:rep3}
P = \bigcap_{i = 1}^{N} \{ x \in \lt^* \; | \langle x, m_i n_i \rangle+a_i \geq 0\},
\end{align} 
where $n_i \in \mathfrak l$ is primitive, and $F_i:=P \cap \{ x \in \lt^* \; | \langle x, m_i n_i \rangle+a_i =0\}$ is the $i$-th facet, to which the integer $m_i$ is attached. Denote $\mathring{P}$ as the interior of the polyhedron $P$, and define $\mathring{M} = \Phi^{-1}(\mathring{P})$.

We fix a basis $\{e_1,\cdots,e_n\}$ of $\lt$ and denote by $K_i =(e_i)_M$ the induced fundamental vector field. Also, we can identify $\lt^* \simeq \R^n$ by using the dual basis $\{e_1^*,\cdots,e_n^*\}$. If we set 
\begin{align*}
\iota_{K_i} \omega = - d x_i,
\end{align*}
the moment map $\Phi$ is $x=(x_1,\cdots, x_n)$.

Let $g(\cdot, \cdot)=\omega(\cdot, J\cdot)$ be the K\"ahler metric, and we set $H_{ij} = g(K_i,K_j)$. Since $H_{ij}$ is $\T^n$-invariant,we can regard $H_{ij}$ as a smooth function $H_{ij}(x)$ on $P$. More intrinsically, $\mathbf{H}:=\{H_{ij}\}$ is a smooth function over $P$ with values in $S^2\lt^*$, i.e., symmetric $2$-form on $\lt^*$. Since $\mathbf{H}$ is positive definite on $\mathring{P}$, we denote $\mathbf{G} : = \mathbf{H}^{-1}$ the inverse matrix on $\mathring{P}$ with values in $S^2\lt$.

It is clear that $\{K_1,\cdots,K_n,JK_1,\cdots,JK_n\}$ forms a commutative basis of $T_p M$ at each point $p \in \mathring{M}$. On $\mathring{M}$, we denote its dual $1$-forms by $\{\theta_1,\cdots,\theta_n,J\theta_1,\cdots,J \theta_n\}$. It is not hard to see that
\begin{align*}
d \theta_i = d (J \theta_i)= \iota_{K_j} J \theta_i=\LL_{K_j} J \theta_i = 0.
\end{align*}
Therefore, $J \theta_i=\Phi^*(\beta)$ for a $1$-form $\beta$ on $\mathring{P}$. Since $J \theta_i$ is closed and $\mathring{P}$ is contractible, we can write 
\begin{align*}
J \theta_i = - d y_i
\end{align*}
for some smooth function $y_i (x)$, defined on $\mathring{P}$ up to an additive constant. Furthermore, we have 
\begin{align} \label{eq:id1}
\frac{d y_i}{dx_j} = G_{ij}(x).
\end{align}
Note that the connection of the torus bundle $\mathring{M} \to \mathring{P}$ has the flat connection $\mathrm{\theta}:=\sum_{i =1}^n \theta_i \otimes e_i$. The K\"ahler form $\omega$ and the K\"ahler metric $g$ on $\mathring{P}$ become
\begin{align}
\label{eq:id2}
\omega = \sum_{i = 1}^n d x_i \wedge \theta_i \quad \text{and} \quad g = \sum_{i,j=1}^n (G_{ij} d x_i \otimes d x_j + H_{ij} \theta_i \otimes \theta_j).
\end{align} 

We fix a point $p_0 \in \mathring{M}$ and identify $\mathring{M}=(\C^*)^n(p_0)=(\C^*)^n$ with holomorphic coordinates $z_i=\exp(y_i+\sqrt{-1} t_i)$, where $(t_1,\cdots, t_n): \mathring{M} \to \lt/\mathfrak{l}=\T^n$ are the angular coordinates with $dt_i=\theta_i$. The functions $\{x_1,\cdots,x_n; t_1,\cdots,t_n\}$ on $\mathring{P} \times \T^n \cong \mathring{M}$ are called \emph{momentum-angle coordinates}. Note that once $p_0$ is chosen, $\{y_i\}$ are determined.

The following result is from \cite{G94}.

\begin{lem}
There exists a smooth, strictly convex function $u(x)$ on $\mathring{P}$, called \emph{symplectic potential} of $g$, satisfying
\begin{align*}
G_{ij} = \frac{\partial^2 u}{\partial x_i \partial x_j}.
\end{align*}
\end{lem}

It follows from \eqref{eq:id1} and the discussion above that we may assume that $y_i=\partial u/\partial x_i$. Since $u$ is strictly convex on $\mathring{P}$, and $y=(y_1,\cdots, y_n)$ sends $\mathring{P}$ onto $\R^n$, we define the Legendre transform of the symplectic potential $u(x)$ as the function $\phi(y) = \phi(y_1,\cdots,y_n)$ satisfying
\begin{align*}
\phi(y(x)) + u(x) = \langle y(x), x\rangle,
\end{align*}
The map $\phi(y): \R^n \to \R$ is called the \emph{K\"ahler potential}, a name justified by the following result from \cite[Theorem 4.1]{G94}.

\begin{lem}
\label{lem:kp}
On $\mathring{M}$, we have
\begin{align*}
\omega = 2\sqrt{-1}\partial \bar \partial \phi \quad \text{and} \quad g = \sum_{i,j=1}^n \lc \frac{\partial^2 \phi}{\partial y_i \partial y_j} d y_i \otimes d y_j + \frac{\partial^2 \phi}{\partial y_i \partial y_j} \theta_i \otimes \theta_j \rc.
\end{align*} 
\end{lem}

Given $P$ as in \eqref{eq:rep3}, we know consider the standard K\"ahler toric orbifold $(M_P^{2n}, \omega_P, J_P, (\C^*)^n, \Phi_P)$ obtained in Theorem \ref{T206}. It follows from \cite[Theorem 1]{Ab01} that its symplectic potential is given by
\begin{align} \label{eq:cansp}
u_P (x) = \frac{1}{2} \sum_{i =1}^N (l_i(x) + a_i ) \log(l_i(x) + a_i),
\end{align}
where $l_i(x)$ is the linear function $l_i(x)=\la m_i n_i,x \ra$. Combining Lemma \ref{lem:kp} and the expression \eqref{eq:cansp}, we have the following result similar to \cite[Theorem 4.5]{G94}. 

\begin{prop}
\label{prop:class}
For the K\"ahler toric orbifold $(M_P^{2n}, \omega_P, J_P, (\C^*)^n, \Phi_P)$, the following conclusions hold.
\begin{enumerate}[label=(\roman*)]
\item The K\"ahler potential $\phi_p$ is given by
\begin{align}
\label{eq:id3}
\phi_P=\frac{1}{2} \sum_{i =1}^N \lc -a_i \log( l_i(x)+a_i) + l_i(x) \rc.
\end{align} 
\item In the orbifold de Rham cohomology $H^2_{\mathrm{dR}}(M_P)$, we have
\begin{align*}
\frac{[\omega_P]}{2\pi}=\sum_{i =1}^N a_i [D_i],
\end{align*}
where $D_i=\Phi_P^{-1}(F_i)$ is the Weil divisor given by the preimage of the $i$-th facet $F_i$ of $P$. 
\end{enumerate}
\end{prop}

\begin{proof}
Since $\phi_P$ is the Legendre transformation of $u_P$, \eqref{eq:id3} is immediate by direct computation. To prove (ii), we only need to show that
\begin{align}
\label{eq:id5}
-\frac{\sqrt{-1}}{2\pi} [\partial \bar \partial \log( l_i(x)+a_i)]= [D_i].
\end{align}
in $H^2_{\mathrm{dR}}(M_P)$. Following \cite[Theorem 6.2]{G94}, we sketch the proof. Let $S^1$ be the circle of $\T^n$ generated by $n_i$. Then, the corresponding moment map, denoted by $\phi$, is $\la n_i, \Phi_P \ra+a_i/m_i$ after a translation. In particular, $\phi \ge 0$ on $M$ and $D_i=\phi^{-1}(0)$.

Thus, \eqref{eq:id5} follows from
\begin{align}
\label{eq:id5a}
-\frac{\sqrt{-1}}{2\pi} [\partial \bar \partial \log \phi]= [D_i].
\end{align}
Indeed, since $D_i$ is a complex suborbifold, for each $p \in D_i$, one can linearize the circle action. That is, there exists a holomorphic local chart $(\tilde U, \Gamma_p)$ around $p$ with complex coordinates $(z_1,\cdots, z_n)$ so that the lift of $U \cap D_i$ is the set where $z_1=0$, and the lift of the circle acts as
\begin{align*}
e^{2\pi \sqrt{-1} \theta}(z_1,z_2,\cdots,z_n)=(e^{2\pi \sqrt{-1} \theta} z_1,z_2,\cdots, z_n).
\end{align*}
In addition, the lift of $\phi$ on $\tilde U$, denoted by $\tilde \phi$, satisfies $\tilde\phi=|z_1|^2 h$, where $h$ is a smooth positive function on $\tilde U$. By considering all such open sets like $U$ and their transitions maps, we conclude that $-\frac{\sqrt{-1}}{2\pi} [\partial \bar \partial \log \phi]$ represents exactly the first Chern class of the orbifold line bundle corresponding to $D_i$. Therefore, \eqref{eq:id5a} holds.
\end{proof}

As a corollary, we prove

\begin{cor} \label{cor:class}
Let $(M^{2n}, \omega,J, (\C^*)^n, \Phi)$ be a K\"ahler toric orbifold with $(\dagger)$. Then, the labeled polyhedron $\Phi(M)$ is determined by the class $[\omega]$ in $H^2_{\mathrm{dR}}(M)$.
\end{cor}

\begin{proof}
We set $P=\Phi(M)$, represented as in \eqref{eq:rep3}, and let $D_i=\Phi^{-1}(F_i)$. First, note that the toric variety $(M^{2n}, J, (\C^*)^n)$ itself determines the fan associated with $P$, and hence all primitive normal vectors $\{n_i\}$. In addition, the isotropic group $\T_x$ is generated by $n_i$, for any $x \in \mathring{D_i}$. Therefore, it follows from Theorem \ref{T206} that the label $m_i$ for the $i$-th facet is determined by the structure group at $x$, for any $x \in \mathring{D_i}$.

Furthermore, it follows from Theorem \ref{thm:to} and Proposition \ref{prop:class} (ii) that
\begin{align*}
\frac{[\omega]}{2\pi}=\sum_{i =1}^N a_i [D_i].
\end{align*}
By \cite[Theorem 4.2.1,Theorem 12.3.2]{CLS11} and Proposition \ref{thm:der}, we have the following short exact sequence
\begin{align*}
0 \longrightarrow \lt^* \overset{i}{\longrightarrow} \mathcal D \longrightarrow H^2_{\mathrm{dR}}(M) \longrightarrow 0,
\end{align*}
where $\mathcal D:=\{\sum_i c_i D_i \mid c_i \in \R\}$, and the map $i$ is given by 
\begin{align*}
i(w)=\sum_i \la w, n_i \ra D_i,
\end{align*}
for $w \in \lt^*$. Therefore, once $[\omega]$ is given, $a_i$ is determined up to an additive term $\la w, n_i \ra$ for some $w \in \lt^*$. Consequently, $P$ can be written, up to a translation, as
\begin{align*}
P = \bigcap_{i = 1}^{N} \{ x \in \lt^* \; | \langle x, m_i n_i \rangle+a_i \geq 0\}.
\end{align*} 
In sum, the associated labeled polyhedron is determined by $[\omega]$.
\end{proof}

For $(M_P^{2n}, \omega_P, J_P, (\C^*)^n, \Phi_P)$, we fix a momentum-angle coordinates $\{x_1,\cdots,x_n; t_1,\cdots,t_n\}$ on $\mathring{M_P}$. When the $\T^n$-invariant, compatible complex structure varies on $M_P$, we compare the corresponding metrics.

Let $J$ be another $\T^n$-invariant, compatible complex structure on $(M_P, \omega_P, \T^n)$. We set $g=\omega_P(\cdot, J \cdot)$. Similar to the discussions above, there exists a smooth, strictly convex function $u$ on $\mathring{P}$ such that
\begin{align}
g= u_{ij} dx_i \otimes dx_j+u^{ij} dt_i \otimes dt_j, \label{eq:kahs}
\end{align} 
where $(u_{ij})=\mathrm{Hess}(u)$ and $(u^{ij})$ its inverse. Note that the invariant complex structure $J$ is given by $J dt_i=-u_{ij} dx_j$. Conversely, any smooth, strictly convex function $u$ on $\mathring{P}$ defines by \eqref{eq:kahs} an invariant metric. Now, the following result was proved in \cite[Theorem 2]{Ab03}, see also \cite[Proposition 1]{ACGT04}.

\begin{thm}
\label{thm:bound}
With the above assumptions, the function $u$ determined in \eqref{eq:kahs} by a $\T^n$-invariant, compatible complex structure $J$ satisfies the following:
\begin{enumerate}[label=(\roman*)]
\item $u-u_P$ is smooth on $P$;
\item The function $\det \lc\mathrm{Hess}\, u \rc \prod_{i=1}^N \lc l_i(x)+a_i\rc$ is positive and smooth on $P$. 
\end{enumerate}
Conversely, any smooth, strictly convex function $u$ on $\mathring{P}$ satisfying (i) and (ii) defines an invariant K\"ahler metric on $M_P$ via \eqref{eq:kahs}. 
\end{thm}

\subsection{Geometry of Ricci shrinker orbifolds}

In this section, we consider the Riemannian Ricci shrinker orbifolds and prove some fundamental estimates. Most estimates and proofs are adapted from those for Ricci shrinkers on manifolds.

\begin{defn}[Ricci shrinker orbifold]
A Ricci shrinker orbifold $(M^n, g, f)$ is a complete Riemannian orbifold (cf. Definition \ref{def:ro}) coupled with a smooth function $f:M \to \R$ such that
\begin{align}\label{eq:rso}
\Rc+\na^2 f=\frac{1}{2} g.
\end{align} 
\end{defn}

As in the manifold case, the function $R+|\na f|^2-f$ is always a constant on a Ricci shrinker orbifold. Therefore, we normalize $f$ by adding a constant so that 
\begin{align}\label{eq:rso1}
R+|\na f|^2=f.
\end{align} 

It is well-known that any Ricci shrinker on a manifold has nonnegative scalar curvature (cf. \cite{CBL07} and \cite{Zh09}). We demonstrate that a Ricci shrinker orbifold also has nonnegative scalar curvature if it is bounded below. Furthermore, the scalar curvature must be positive unless the orbifold is flat.

\begin{prop}\label{prop:sca}
Let $(M^n,g,f)$ be a Ricci shrinker orbifold such that $\inf_{|M|} R>-\infty$. Then $R>0$ everywhere unless $(M,g)$ is isometric to the flat cone $\R^n/\Gamma$, where $\Gamma \leq \mathrm{O}(n)$ is a finite subgroup.
\end{prop}

\begin{proof}
We first prove $R \ge 0$ on $M$. By our assumption, we assume $\inf_{|M|} R>-B$ for a positive constant $B$. We set $\bar f=f+B$, then by \eqref{eq:rso} and \eqref{eq:rso1},
\begin{align}
\Rc+\na^2 \bar f=\frac{1}{2} g \quad \text{and} \quad |\na \bar f|^2 \le \bar f. \label{eq:rso2}
\end{align} 
In particular, it implies that $\bar f \ge 0$ and $\sqrt{\bar f}$ is $1/2$-Lipschitz on $|M|$. We fix a point $p \in M_{\mathrm{reg}}$. Then, it follows from the geodesically convexity of $M_{\mathrm{reg}}$ and the same argument as \cite[Proposition 2.1]{CZ10} (see also \cite[Lemma 2.1]{HM11} and \cite[Lemma 1]{LW20}) that for any $x \in M_{\mathrm{reg}}$,
\begin{align}
\frac{d(p,x)}{2}+\frac{1}{3}-2n \le \sqrt{\bar f(p)}+\sqrt{\bar f(x)}. \label{eq:rso3}
\end{align} 
Since $\bar f$ is smooth and $M_{\mathrm{reg}}$ is dense, it is clear that \eqref{eq:rso3} holds for any $x \in |M|$. In particular, $\bar f$ is proper and bounded below.

Next, we recall the definition of weighted Laplacian $\Delta_f:=\Delta-\la \na f, \cdot \ra$. The following identities then hold:
\begin{align}
\Delta_f \bar f=&\frac{n}{2}-\bar f, \label{eq:rso4}\\
\Delta_f R=&R-2|\Rc|^2. \label{eq:rso5}
\end{align} 
Note that both identities are well-known for Ricci shrinkers on manifolds (see, for instance, \cite{PW10}), and hence, they hold on Ricci shrinker orbifolds if one uses the local charts. We take a smooth decreasing cutoff function $\eta(t)$ on $\R$ such that $\eta=1$ on $(-\infty,1]$ and $\eta=0$ on $[2,\infty)$. Additionally, we require that
\begin{align} 
\frac{(\eta')^2}{\eta}+|\eta''| \le C \label{eq:rso6}
\end{align} 
for a universal constant $C$.

For any constant $A>1$, we set $\varphi=\eta(\bar f/A)$ and $u=\varphi R$, which are smooth functions on $M$. From direct computation by using \eqref{eq:rso2}, \eqref{eq:rso4} and \eqref{eq:rso5}, we have
\begin{align}
\Delta_f u=&\varphi \Delta_f R+R\Delta_f\phi+2\la \na R, \na \varphi \ra \notag \\ 
=& \varphi (R-2|\Rc|^2)+R \lc \eta' \frac{n/2-\bar f}{A}+\eta'' \frac{|\na \bar f|^2}{A^2} \rc+2\frac{\eta'}{A} \la \na R, \na f \ra. \label{eq:rso7}
\end{align} 
Suppose that $u$ takes its minimum value at $q \in |M|$. By the maximum principle, considering the local chart around $q$ if necessary, we have at $q$:
\begin{align*}
\Delta_f u \ge 0 \quad \text{and} \quad \na u=\varphi \na R+\frac{\eta'}{A} R\na \bar f=0.
\end{align*}
Combining with \eqref{eq:rso6} and \eqref{eq:rso7}, we obtain at $q$:
\begin{align}\label{eq:rso8}
0\le \varphi (R-2R^2/n)+R \lc \eta' \frac{n/2-\bar f}{A}+\eta'' \frac{|\na \bar f|^2}{A^2}-2\frac{(\eta')^2}{\varphi} \frac{|\na \bar f|^2}{A^2} \rc. 
\end{align} 
Notice that $|\na \bar f|^2 \le \bar f \le 2A$ on the support of $u$, and $\eta' \le 0$. If $u(q) \le 0$, it follows from \eqref{eq:rso6} and \eqref{eq:rso8} that
\begin{align*}
u(q) \ge -\frac{C}{A},
\end{align*}
for a constant $C$ depending on $n$. By letting $A \to +\infty$, we conclude that $R \ge 0$ everywhere.

If there exists $x_0 \in |M|$ such that $R(x_0)=0$, it follows from \eqref{eq:rso5} and the strong maximum principle that $R \equiv 0$ on $M$. By \eqref{eq:rso5} again, we conclude that $\Rc \equiv 0$ and hence 
\begin{align}\label{eq:rso9}
\na^2 f=\frac{g}{2}
\end{align} 
on $M$. As discussed above, $f$ is proper and bounded below, so there exists a minimum point $p$, which may be a singular point. Let $r:=d(p,\cdot)$, and we claim
\begin{align}\label{eq:rso10}
f=\frac{r^2}{4}+f(p).
\end{align}
Indeed, for any $x \in M_{\mathrm{reg}}$, there exists a unit minimizing geodesic $\gamma(t): t\in [0,r(x)]$ starting from $p$ to $x$ such that $\gamma((0,r(x)]) \subset M_{\mathrm{reg}}$. Since $p$ is a minimum point of $f$, it follows immediately from \eqref{eq:rso9} that \eqref{eq:rso10} holds at $x$. Since $M_{\mathrm{reg}}$ is dense, \eqref{eq:rso10} holds everywhere.

Given that $|\na f|^2=f$, it follows from \eqref{eq:rso9} and \eqref{eq:rso10} that, away from $p$,
\begin{align*}
\mathcal L_{\na r} g=\frac{2}{r} (g-dr^2),
\end{align*}
which implies that for any vector field $U,V$ such that $[U,\na r]=[V,\na r]=0$, we have
\begin{align*}
\partial_r (g(U,V))=(\mathcal L_{\na r} g)(U,V)=\frac{2}{r} g(U,V).
\end{align*}
Therefore, if we denote the level set $f=1/4$ by $(\Sigma, g_{\Sigma})$, which is a Riemannian orbifold, then
\begin{align*}
g=dr^2+r^2 g_{\Sigma},
\end{align*}
which means $(M,g)$ is a cone orbifold. Since the tangent cone at $p$ is a flat cone $\R^n/\Gamma$, where $\Gamma$ is the structure group at $p$, we conclude immediately that $(M,g)$ is isometric to $\R^n/\Gamma$.

In sum, the proof is complete.
\end{proof}

\begin{rem}
For a Ricci shrinker orbifold with a compact singular set $|M|_{\mathrm{sing}}$, it is also true that $R \ge 0$. Indeed, this fact can be proved by the same argument as in \emph{\cite[Theorem 1.3]{Zh09}}, by applying the maximum principle for $\phi R$, where $\phi$ is a cutoff function constructed by using the distance function. Note that one can assume that the singular set is contained in $\{\phi=1\}$. We conjecture that Proposition \ref{prop:sca} still holds without the assumption $\inf_{|M|} R>-\infty$.
\end{rem}

For a Ricci shrinker orbifold, since its regular part is geodesically convex, the following result can be proved as in \cite[Proposition 2.1]{CZ10} or \cite[Lemma 2.1]{HM11}, by using Proposition \ref{prop:sca}.

\begin{prop}\label{prop:estf}
Let $(M^n,g,f)$ be a Ricci shrinker orbifold such that $\inf_{|M|} R>-\infty$. Then there exists a point $p \in |M|$ where $f$ attains its miminum and $f$ satisfies the quadratic growth estimate
\begin{align*}
\frac{1}{4} \lc d(p,x)-5n \rc_+^2 \le f(x) \le \frac{1}{4} \lc d(p,x)+\sqrt{2n} \rc^2,
\end{align*}
for all $x \in |M|$, where $a_+:=\max\{a,0\}$.
\end{prop}

By using Proposition \ref{prop:estf}, we can follow the same proof as in \cite[Theorem 1.2]{CZ10} or \cite[Lemma 2.2]{HM11} verbatim to show the following volume estimate.

\begin{prop}\label{prop:volume}
Let $(M^n,g,f)$ be a Ricci shrinker orbifold such that $\inf_{|M|} R>-\infty$. Then, there exists a constant $C=C(n)>0$ such that
\begin{align*}
|B(p,r)| \le C r^n,
\end{align*}
where $p$ is a minimum point of $f$.
\end{prop}

Next, we prove the following result, which is also well-known for Ricci shrinkers on manifolds (cf. \cite{Lott03} \cite{WW07}).

\begin{prop}\label{prop:fund}
Let $(M^n,g,f)$ be a Ricci shrinker orbifold such that $\inf_{|M|} R>-\infty$. Then, the orbifold fundamental group $\pi_1^{\mathrm{orb}}(M)$ is finite.
\end{prop}

\begin{proof}
Let $\rho: (\tilde M, \tilde g, \tilde f ) \to (M, g,f)$ be the universal cover, which is also a Ricci shrinker orbifold by using the pullback metric $\tilde g=\rho^* g$ and potential function $\tilde f=\rho^* f$. Let $p$ be a minimum point of $f$ and fix $\tilde p \in \rho^{-1}(p)$. Then, $\tilde p$ is also a minimum point of $\tilde f$, so the estimates in Proposition \ref{prop:estf} also hold for $\tilde f$, that is, 
\begin{align*}
\frac{1}{4} \lc d(\tilde p,x)-5n \rc_+^2 \le \tilde f(x) \le \frac{1}{4} \lc d(\tilde p,x)+\sqrt{2n} \rc^2,
\end{align*}
for any $x \in |\tilde M|$. Therefore, it implies that the map $\rho$ is proper. Since $\rho$ is a covering map, the preimage of any regular point in $|M|$ has a finite number of points. Consequently, $\pi_1^{\mathrm{orb}}(M)$ is finite.
\end{proof}

As a corollary, we prove

\begin{cor}\label{cor:dr1}
Let $(M^n,g,f)$ be a Ricci shrinker orbifold such that $\inf_{|M|} R>-\infty$. Then, 
\begin{align*}
H^1_{\mathrm{dR}}(M)=0.
\end{align*}
\end{cor}
\begin{proof}
Since there exists a natural epimorphism $\pi_1^{\mathrm{orb}}(M) \to \pi_1(|M|)$ (cf. \cite[Proposition 2.5]{BMP03}), we conclude by Proposition \ref{prop:fund} that $\pi_1(|M|)$ is also finite. Therefore, the conclusion follows from Theorem \ref{thm:der}.
\end{proof}

\section{Proof of the main theorems}

\subsection{K\"ahler Ricci shrinkers on toric orbifolds}

In this section, we consider Ricci shrinker orbifolds with extra structures and symmetries. First, we have the following definition.

\begin{defn}[Toric K\"ahler Ricci shrinker orbifold]
A \emph{toric K\"ahler Ricci shrinker orbifold} $(M^{2n}, g, J, f, \T^n)$ is a Ricci shrinker orbifold $(M,g,f)$ with a compatible complex structure $J$ so that $\omega:=g(J\cdot, \cdot)$ is closed, together with an effective, holomorphic and isometric action of $\T^n$.
\end{defn}

In particular, it follows from the equation \eqref{eq:rso} that $\na f$ is a real holomorphic vector field and $J \na f$ is a Killing field. For a toric K\"ahler Ricci shrinker orbifold with bounded scalar curvature, we prove

\begin{prop}
\label{prop:kros}
Let $(M^{2n}, g, J, f, \T^n)$ be a toric K\"ahler Ricci shrinker orbifold with bounded scalar curvature. Then the following assertions hold:
\begin{enumerate}[label=(\roman*)]
\item The $\T^n$-action preserves the potential function $f$.
\item The $\T^n$-action on the symplectic form $\omega:=g(J\cdot, \cdot)$ is Hamiltonian.
\item $J \na f \in \lt$, the Lie algebra of $\T^n$.
\item The action of $\T^n$ extends to a holomorphic $(\C^*)^n$-action.
\end{enumerate}
\end{prop}
\begin{proof}
(i): We claim that for any $Y \in \mathfrak l$, the lattice of $\lt$, we have $Yf\equiv 0$. Indeed, we denote the family of isometries $e^{2\pi t Y}$ on $M$ by $\phi_t$, where $t \in [0,1]$. Since $\phi_t$ is isometric, it follows from \eqref{eq:rso} that
\begin{align*}
\Rc+\na^2( \phi_t^* f)=\frac{1}{2} g.
\end{align*}
Taking the derivative with respect to $t$ at $t=0$, we conclude that
\begin{align} \label{eq:k01a}
\na^2 (Yf)=0.
\end{align}

If $(M, g)$ does not split off a line, it follows from \eqref{eq:k01a} that $Yf$ must be a constant. Since $f$ has a minimum point by Proposition \ref{prop:estf}, it is immediate that $Yf \equiv 0$. On the other hand, if $(M, g,f)$ indeed splits off a line, we assume that $(M, g,f)$ is isometric to $(N \times \R^k, g_N \times g_E,f_N+|x|^2/4)$, where $(N,g_N,f_N)$ is a Ricci shrinker orbifold that does not split off a line, and $x$ is the coordinate of $\R^k$. Since $Y$ belongs to the lattice, we conclude that $Y$ is tangent to $N$, and hence $Yf=Yf_N$. Therefore, the same argument yields $Yf\equiv 0$. Since $\lt$ is generated by $\mathfrak l$, we conclude that $Yf\equiv 0$ for any $Y \in \lt$.

(ii): Since the $\T^n$-action is holomorphic and isometric, it is also symplectic with respect to $\omega$. Let $\{Y_1,\cdots,Y_n\}$ be a basis of $\lt$. Since 
\begin{align*}
\mathcal L_{Y_i} \omega=d(\iota_{Y_i} \omega)=0,
\end{align*}
it follows from Corollary \ref{cor:dr1} that there exists a smooth function $h_i$ such that
\begin{align*}
-\iota_{Y_i} \omega=d h_i.
\end{align*}
For a general $Y=\sum c_i Y_i \in \lt$, we define the map $\Phi: M \to \lt^*$ by
\begin{align*}
\la \Phi, Y\ra=\sum c_i h_i.
\end{align*}
To show that the $\T^n$-action is Hamiltonian, we only need to prove that $\Phi$ is invariant under $\T^n$, which is equivalent to
\begin{align} \label{eq:k02}
\omega(Y_i, Y_j)\equiv 0.
\end{align}
By direct computation, we have
\begin{align*}
d(\omega(Y_i, Y_j))= d\iota_{Y_j} \iota_{Y_i} \omega=\mathcal L_{Y_j} \iota_{Y_i} \omega-\iota_{Y_i} \mathcal L_{Y_j} \omega=\iota_{[Y_j,Y_i]} \omega=0.
\end{align*}
Consequently, it implies that $\omega(Y_i, Y_j) \equiv c$ for a constant $c$. By Proposition \ref{prop:estf}, the level set $\{f=L\}$ is compact and nonempty for all large constant $L$. We assume $h_j$ takes its minimum on $\{f=L\}$ at $q$. By (i), $\T^n$ preserves $\{f=L\}$, so $\omega(Y_i, Y_j)=Y_i h_j=0$ at $q$. Therefore, \eqref{eq:k02} holds.

(iii): As before, we set $\{Y_1,\cdots,Y_n\}$ to be a basis of $\lt$. Since the $\T^n$-action is effective, there exists an open dense set $U \subset M_{\mathrm{reg}}$ on which $\T^n$ has no fixed point. It is clear from (ii) that $T_x M=\lt \oplus J \lt$ orthogonally at each point $x \in U$. Moreover, since each $Y_i$ preserves $f$, we can set
\begin{align*}
J\na f=\sum_i b_i Y_i,
\end{align*}
on $U$, where $b_i$ is a smooth function. By direct computation, for any $k$,
\begin{align*}
\sum_i Y_k(b_i) Y_i=[Y_k ,\sum_i b_i Y_i]=[Y_k, J\na f]=J[Y_k, \na f]=0,
\end{align*}
where we have used the fact that $Y_k$ is a holomorphic Killing field which preserves $f$. Similarly, we have
\begin{align*}
\sum_i JY_k(b_i) Y_i=[JY_k ,\sum_i b_i Y_i]=[JY_k, J\na f]=-[Y_k, \na f]=0.
\end{align*}
Therefore, we conclude that $b_i$ is a constant on $U$. By continuity, it is immediate that $J \na f \in \lt$.

(iv): We only need to prove that $J Y$ is complete for any $Y\in \lt$. By our assumption that the scalar curvature $R$ is bounded, it follows from \eqref{eq:rso1} and Proposition \ref{prop:estf} that the set of critical points of $f$ is compact. Since $[JY,\na f]=0$, it follows from \cite[Lemma 2.34]{CDS24} that $JY$ is complete.
\end{proof}

As a corollary, we have

\begin{cor}\label{cor:cha} 
Let $(M^{2n}, g, J, f, \T^n)$ be a toric K\"ahler Ricci shrinker orbifold with bounded scalar curvature. Then, it is a K\"ahler toric orbifold with $(\dagger)$.
\end{cor}
\begin{proof}
By Proposition \ref{prop:kros} (ii)(iv), we have shown that $(M,\omega,\T^n, \Phi)$ has a Hamiltonian $\T^n$-action which extends to a holomorphic $(\C^*)^n$-action, where $\omega=g(J\cdot, \cdot)$, and $\Phi$ is a moment map.

To show $(\dagger)$, we first note that the set of critical points of $f$ is compact since $R$ is bounded. Then, it follows from Proposition \ref{prop:kros} (iii) that there are only finitely many fixed points of $\T^n$, since all fixed points are isolated. In addition, it is clear that
\begin{align*}
d\la \Phi, J\na f \ra=-\iota_{J\na f} \omega=g(\cdot, \na f)=df,
\end{align*}
and hence by Proposition \ref{prop:estf}, $\la \Phi, J\na f \ra$ is proper and bounded below.
\end{proof}

As a K\"ahler toric orbifold, the labeled polyhedron is determined for a toric K\"ahler Ricci shrinker orbifold with bounded scalar curvature.

\begin{lem}\label{lem:class2} 
Let $(M^{2n}, g, J, f, \T^n)$ be a toric K\"ahler Ricci shrinker orbifold with bounded scalar curvature. Then, the associated labeled polyhedron $P$ is independent of the metric $g$.
\end{lem}

\begin{proof}
By Corollary \ref{cor:class}, we only need to prove that the class $[\omega]$ is determined. This is true since by the Ricci shrinker equation \eqref{eq:rso}
\begin{align*}
\frac{[\omega]}{4\pi}=-\mathcal K,
\end{align*}
in $H^2_{\mathrm{dR}}(M)$, where $\mathcal K$ is the canonical class.
\end{proof}

For a toric K\"ahler Ricci shrinker orbifold $(M^{2n}, g, J, f, \T^n)$ with bounded scalar curvature, it follows from \cite[Theorem 8.2.3]{CLS11} that
\begin{align*}
-\mathcal K=\sum_i [D_i].
\end{align*}
Therefore, by the proof of Corollary \ref{cor:class}, we conclude that the image of a moment map can be expressed as
\begin{align*}
P = \bigcap_{i = 1}^{N} \{ x \in \lt^* \; | \langle x, m_i n_i \rangle+2 \geq 0\}.
\end{align*} 

In general, for any positive integer $m$, we consider the rescaled metric $g'=\frac{m}{2}g$ and $\omega'=g'(J\cdot, \cdot)$. It is clear that
\begin{align*}
\Rc(g')+\na_{g'}^2 f=\frac{1}{m} g'.
\end{align*} 
Therefore, the image of the moment map with respect to $\omega'$ can be expressed as 
\begin{align}\label{eq:rep4}
P'=\frac{2}{m}P= \bigcap_{i = 1}^{N} \{ x \in \lt^* \; | \langle x, m_i n_i \rangle+m \geq 0\}.
\end{align} 
Since $P'$ is a rational, simple polyhedron, we conclude from \eqref{eq:rep4} that all vertices belong to $\mathfrak l^*$, if $m$ is appropriately chosen. In other words, $P'$ is a lattice polyhedron. Consequently, the following result is immediate from Theorem \ref{thm:ka} and \cite[Proposition 7.2.9]{CLS11}.

\begin{cor}
Let $(M^{2n}, g, J, f, \T^n)$ be a toric K\"ahler Ricci shrinker orbifold with bounded scalar curvature. Then, $(M, J)$ is a quasiprojective toric variety.
\end{cor}

\subsection{Uniqueness of solutions}

\subsubsection*{Proof of Theorem \ref{thm:A}}

In this section, we prove Theorem \ref{thm:A}. Let $(M^{2n}, g, J, f, \T^n)$ be a toric K\"ahler Ricci shrinker orbifold with bounded scalar curvature. Then, it follows from Corollary \ref{cor:cha} that $(M, \omega, J, (\C^*)^n, \Phi)$ is a K\"ahler toric orbifold with $(\dagger)$, where $\omega=g(J\cdot, \cdot)$, and $\Phi$ is a moment map. Moreover, we set $P=\Phi(M)$ to be the labeled polyhedron, which is defined independently of $g$ by Lemma \ref{lem:class2} and can be normalized and represented as 
\begin{align}\label{eq:rep3a}
P = \bigcap_{i = 1}^{N} \{ x \in \lt^* \; | \langle x, m_i n_i \rangle+2 \geq 0\}.
\end{align} 

For the standard K\"ahler toric orbifold $(M_P, \omega_P, J_P, (\C^*)^n,\Phi_P)$, we fix a momentum-angle coordinates $\{x_1,\cdots,x_n; t_1,\cdots,t_n\}$ on $\mathring{P} \times \T^n \cong \mathring{M_P}$. It follows from Theorem \ref{T204} that there exists a $\T^n$-invariant symplectomorphism $\varphi:M_P \to M$. We set $J_1:=\varphi^* J$, $f_1:=\varphi^* f$ and $g_1:=\omega_P(\cdot, J_1 \cdot)$. Moreover, we denote $J_1\na^{g_1} f_1$ by $X$. Note that $X$ is a real holomorphic vector field with $X \in \lt$ by Proposition \ref{prop:kros} (iii). 

Then, it is clear that $(M_P, g_1, J_1, f_1, \T^n)$ is a toric K\"ahler Ricci shrinker orbifold with bounded scalar curvature. Let $u$ be the symplectic potential on $P$, with respect to $g_1$, together with its Legendre transform by $\phi$, i.e., $\phi(y)+u(x)=\la y,x \ra$, where $y_i=\partial u/\partial x_i$. Notice that since the $\T^n$-action also extends to $(\C^*)^n$-action on $M_P$, with respect to $J_1$, we have $\na u: P \to \R^n$ is surjective, that is, $y=(y_1,\cdots, y_n)$ are coordinates of $\R^n$.

Now, one can rewrite the Ricci shrinker equation \eqref{eq:rso} in terms of the symplectic potential.

\begin{prop} \label{prop:symp}
With the above assumptions, the symplectic potential $u$, after adding a constant if necessary, satisfies:
\begin{align}\label{eq:m01}
u_i x_i-u(x)-\log \det(u_{ij})=\la b_X,x \ra,
\end{align} 
on $\mathring{P}$, where $b_X \in \lt$ is a constant vector.
\end{prop}

\begin{proof}
From Lemma \ref{lem:kp}, it follows that
\begin{align}\label{eq:m02}
\omega_P = 2\sqrt{-1}\partial \bar \partial \phi,
\end{align} 
where the underlying complex structure is $J_1$. Thus, the Ricci form of $g_1$, denoted by $\Rc_1$, can be expressed as
\begin{align}\label{eq:m03}
\Rc_1=-\sqrt{-1} \partial \bar \partial\log\det(\phi_{ij}).
\end{align} 
Moreover, we have 
\begin{align}\label{eq:m04}
-\mathcal L_{J_1X} \omega_P=-d \iota_{J_1X} \omega_P=-d \lc g_1(X, \cdot) \rc=dd^c f_1,
\end{align} 
where $d^c=J_1 d$. The Ricci shrinker equation \eqref{eq:rso} can be written as
\begin{align*}
\Rc_1-\frac{1}{2} \mathcal L_{J_1X} \omega_P-\frac{1}{2} \omega_P=0.
\end{align*} 
Therefore, from \eqref{eq:m02}, \eqref{eq:m03} and \eqref{eq:m04}, we have
\begin{align*}
\partial \bar \partial \lc -\log \det(\phi_{ij})+f_1-\phi \rc \equiv 0,
\end{align*} 
and thus
\begin{align}\label{eq:m05}
-\log \det(\phi_{ij})+f_1-\phi=\la a, y \ra+c,
\end{align} 
for a constant vector $a=(a_1,\cdots,a_n)$ and a constant $c$. Applying the Legendre transformation to \eqref{eq:m05}, we obtain
\begin{align}\label{eq:m06}
u_ix_i-u(x)+a_i u_i-\log \det(u_{ij})+c=f_1.
\end{align} 
Since $f_1$ is a Hamiltonian function of $X$, it follows from Proposition \ref{prop:kros} (iii) that there exists a vector $b_X \in \lt$ and constant $c'$ such that $f_1=\la b_X, x\ra+c'$. Therefore, after absorbing all constants into $u$, we obtain from \eqref{eq:m06}
\begin{align}\label{eq:m07}
u_ix_i-u(x)+a_i u_i-\log \det(u_{ij})= \la b_X, x \ra.
\end{align} 
Now, we claim that the vector $a \equiv 0$. Indeed, we set $L_k(x)=\la x,m_kn_k\ra+2$. Then, we can rewrite \eqref{eq:m07} as
\begin{align}\label{eq:m08}
(u_P)_ix_i-u_P(x)+\sum_{k=1}^N \log L_k+a_iu_i=\log\lc \det(u_{ij})\prod_{k=1}^N L_k \rc+(u_P-u)_ix_i+u(x)-u_P(x)+ \la b_X, x \ra.
\end{align} 
It follows from Theorem \ref{thm:bound} that the right-hand side of \eqref{eq:m08} is a smooth function on $P$. On the other hand, the term $(u_P)_ix_i-u_P(x)$, which is the Legendre transformation of $u_P$, satisfies by Proposition \ref{prop:class} (i)
\begin{align}\label{eq:m09}
(u_P)_ix_i-u_P(x)=-\sum_{k=1}^N \log L_k+\frac{1}{2} \sum_{k=1}^N \la x,m_k n_k \ra.
\end{align} 
Combining \eqref{eq:m08} with \eqref{eq:m09}, we conclude that the term $a_iu_i$, and hence $a_i (u_P)_i$ are smooth on $P$. However, we know that
\begin{align}\label{eq:m10}
a_i (u_P)_i=\frac{1}{2} a_i \sum_{k=1}^{N} (L_k)_i \lc \log L_k+1 \rc.
\end{align} 
Note that $(L_k)_i$ is constant. Therefore, \eqref{eq:m10} forces $L_k(a)=2$ for any $k$ and hence
\begin{align*}
\la n_k, a \ra=0.
\end{align*} 
Since $P$ is rational and simple, it implies that $a \equiv 0$. In sum, the proof is complete.
\end{proof}

Next, we show that the vector $b_X \in \lt$ in the equation \eqref{eq:m01} is independent of the choice of the Ricci shrinker metric $g_1$. To achieve this, we consider the following weighted volume functional as in \cite{CDS24} and \cite{Ci22}. 
\begin{defn}[Weighted volume functional] \label{def:wvf}
Given a polyhedron as in \eqref{eq:rep3a}, the weighted volume functional is defined for $b \in \lt$ by 
\begin{align*}
F(b) = \int_P e^{- \langle b,x \rangle} \,dx.
\end{align*}
\end{defn}
It follows from the representation \eqref{eq:rep3a} that $P$ contains $0$ as an interior point. Now, we have

\begin{lem} \label{lem:fun}
$F(b)$ is finite if and only if $b \in \mathrm{int}(C'(P))$, the interior of the dual of the asymptotic cone of $P$. In addition, $F(b)$ is convex and proper on $\mathrm{int}(C'(P))$.
\end{lem}

\begin{proof}
The proof follows verbatim from \cite[Proposition 3.1]{Ci22}.
\end{proof}

By taking the differential, the unique critical point $b$ of $F(b)$, whose existence is guaranteed by Lemma \ref{lem:fun}, is characterized by
\begin{align*}
\frac{\partial}{\partial b^j} F = - \int_P x_j e^{- \langle b,x \rangle} dx=0.
\end{align*}

For later applications, we then prove
\begin{lem}\label{lem:est1}
For any integer $m \ge 0$, we have
\begin{align*}
\int_P |x^m| e^{ - \langle b_X,x \rangle} dx <\infty.
\end{align*}
\end{lem}
\begin{proof}
Since $b_X \in \mathrm{int}(C'(P))$, there exists a constant $\ep>0$ such that $\la b_X, w \ra \ge \ep |w|$ for any $w \in C(P)$. By Lemma \ref{lem:mink}, we can decompose $P=V+C(P)$, where $V$ is the polytope that is the convex hull of all vertices of $P$. Therefore, we conclude
\begin{align*}
\la b_X, x \ra \ge \ep |x|-C
\end{align*}
for any $x \in P$. From this, the conclusion is immediate. 
\end{proof}

Now, one can prove

\begin{lem}
The constant vector $b_X$ in \eqref{eq:m01} is the critical point of $F$, which is determined independently of $g_1$.
\end{lem}
\begin{proof}
We first show that $F(b_X)$ is finite. Under the momentum-angle coordinates, we have
\begin{align*}
\int_P e^{ - \langle b_X,x \rangle} dx=(2 \pi)^{-n} \int_{\mathring{P} \times \T^n} e^{-\langle b_X,x \rangle} d x dt =(2 \pi)^{-n} \int_{(\C^*)^n} e^{-f_1+c} \omega_P^n=C \int_{M_P} e^{-f_1} \,dV_{g_1},
\end{align*}
where we have used the fact that $\omega = \sum_i d x_i \wedge dt_i$ from \eqref{eq:id2} and $f_1=\la b_X, x \ra+c$ as shown before. Therefore, it follows from Proposition \ref{prop:estf} and Proposition \ref{prop:volume} that $F(b_X)$ is finite.

Furthermore, since $\phi$ is the Legendre transformation of $u$, it follows from \eqref{eq:m01} that
\begin{align} \label{eq:410}
\int_P x_j e^{ - \langle b_X,x \rangle} dx= \int_{\R^n} \frac{\partial\phi}{\partial y_j} e^{- \phi} dy = - \int_{\R^n} \frac{\partial e^{-\phi}}{\partial y_j} dy.
\end{align}
Notice that by Lemma \ref{lem:est1}, the first integral above is well-defined. Since $P$ contains $0$, it follows from \cite[Lemma 2.6]{BB13} that $\phi$ satisfies
\begin{align}\label{eq:411}
\phi(y) \ge C^{-1} |y|-C
\end{align}
for a constant $C>1$. Applying the integration by parts, \eqref{eq:410} yields that for any $1 \le j \le n$,
\begin{align*}
\int_P x_j e^{ - \langle b_X,x \rangle} dx=0,
\end{align*}
which finishes the proof.
\end{proof}

Suppose now $(M, g', J, f', \T^n)$ is another toric K\"ahler Ricci shrinker orbifold with bounded scalar curvature. As before, we choose a $\T^n$-invariant symplectomorphism $\psi:M_P \to M$ and set $J_2:=\psi^* J$, $f_2:=\psi^* f$ and $g_2:=\omega_P(\cdot, J_2 \cdot)$. Moreover, we denote the symplectic potential of $g_2$ by $u'$. Clearly, $u'$ satisfies as \eqref{eq:m01}
\begin{align}\label{eq:n01}
u'_i x_i-u'(x)-\log \det(u'_{ij})=\la b_X,x \ra.
\end{align} 
To prove the uniqueness, we first set up the moduli space of symplectic potentials with respect to \eqref{eq:n01}.

\begin{defn}[Space $\mathcal E$] \label{def:e}
The moduli space $\mathcal E$ consists of all smooth function $v$ satisfying
\begin{enumerate}[label=(\roman*)]
\item $v$ is a smooth and strictly convex function on $\mathring{P}$;

\item $v-u_P$ is a smooth function on $P$;

\item $\na v: P \to \R^n$ is surjective;

\item $\displaystyle \int_P |v| e^{-\la b_X, x \ra} \,dx <\infty$.
\end{enumerate}
\end{defn}

We first show 

\begin{lem}
The symplectic potentials $u$ and $u'$ belong to $\mathcal E$.
\end{lem}
\begin{proof}
(i) (ii) and (ii) in the Definition \ref{def:e} are clear, so we only need to prove (iv) for $u$. First, we notice that since $\na u: P \to \R^n$ is surjective, $u$ has a minimum point in $\mathring{P}$. Since $\int_P e^{-\la b_X, x \ra} \,dx <\infty$, we only need to prove
\begin{align} \label{eq:412}
\int_P u e^{ - \langle b_X,x \rangle} dx <\infty.
\end{align}

Recall that $\phi(y)$ is the Legendre transformation of $u$ so that $\phi(y)$ grows at least linearly in $|y|$ by \eqref{eq:411}. We compute
\begin{align*}
\int_P u e^{ - \langle b_X,x \rangle} dx=\int_{\R^n }(\langle \nabla \phi,y \rangle - \phi) e^{-\phi} \,d y \leq \int_{\R^n} \langle \nabla \phi , y \rangle e^{- \phi} \,d y + C.
\end{align*}
On the other hand, by using the polar coordinates and applying the integration by parts, we have
\begin{align*}
\int_{\R^n} \langle \nabla \phi, y \rangle e^{- \phi} d y = \int_{S^{n-1}} \int_0^{\infty} r^n \frac{\partial \phi}{\partial r} e^{- \phi} d r d \Theta = \frac{n}{2} \int_{S^{n-1}}\int_0^{\infty} r^{n-1} e^{- \phi} dr d\Theta = n \int_{\R^n} e^{-\phi} d y.
\end{align*}
Therefore, \eqref{eq:412} holds.
\end{proof}

Now, we can state the uniqueness result in the space $\mathcal E$.

\begin{prop} \label{prop:uniq}
Up to an addition of a linear function, there is at most one solution $v \in \mathcal E$ on $\mathring{P}$ of
\begin{align}\label{eq:p01}
v_i x_i-v(x)-\log \det(v_{ij})= \la b_X, x \ra.
\end{align}
\end{prop}

\begin{proof}
This proof follows from \cite[Theorem 3.11]{Ci22} verbatim, and we sketch it for readers' convenience. 

Suppose $v_0,v_1 \in \mathcal E$ are two solutions to $\eqref{eq:p01}$. We consider the geodesic of the symplectic potentials $v_t=(1-t)v_0+tv_1$ for $t \in [0,1]$. It is straightforward that $v_t$ satisfies Definition \ref{def:e} (i) (ii) and (iv). Moreover, (iii) also holds, as shown in \cite[Lemma 3.6]{Ci22}. Therefore, $v_t \in \mathcal E$ for any $t \in [0,1]$. 

Denoting the Legendre transformation of $v_t$ by $\phi_t$, we define
\begin{align*}
\mathcal{D}_1(v_t) = \int_{\R^n} e^{-\phi_t(y)}\, dy
\end{align*}
and the Ding functional:
\begin{align*}
\mathcal{D}(v_t) = \frac{1}{ \mathcal{D}_1(v_t)} \int_P v_t e^{-\la b_X, x \ra}\,dx -\log \mathcal{D}_1(v_t).
\end{align*}

We first assume $w:=v_1-v_0$ is compactly supported. By the general formula $\dot{v_t}(x)+\dot{\phi_t}(\na v_t(x))=0$, a direct computation shows that
\begin{align*}
\left. \frac{d}{dt} \right \vert_{t=0} \mathcal{D}_1(v_t) = \int_{P} w e^{-\la b_X, x \ra}\, dx
\end{align*}
and hence
\begin{align}\label{eq:p02}
\left. \frac{d}{dt} \right \vert_{t=0} \mathcal{D}(v_t) =0.
\end{align}
Similarly, we have
\begin{align}\label{eq:p03}
\left. \frac{d}{dt} \right \vert_{t=1} \mathcal{D}(v_t) =0.
\end{align}

By the convexity and rigidity of the Ding functional (cf. \cite[Proposition 2.14]{BB13} and \cite[Proposition 3.9]{Ci22}), it follows from \eqref{eq:p02} and \eqref{eq:p03} that $v_1=v_0+L$, where $L$ is a linear function on $P$.

In the general case when $w$ is not compactly supported, one can use a sequence of compactly supported functions to approximate $w$ and the value of $\mathcal D(v_t)$; see \cite[Lemma 3.10]{Ci22} for details.

In sum, the proof is complete.
\end{proof}

If we apply Proposition \ref{prop:uniq} to $u$ and $u'$, we conclude that $u=u'+L$ for a linear function $L$ on $P$. Therefore, under the momentum-angle coordinates $\{x_1,\cdots,x_n; t_1,\cdots,t_n\}$, it follows from \eqref{eq:kahs} that $g_1=g_2$ on $\mathring{M_P}$. Since both $g_1$ and $g_2$ are complete, it implies that $g_1=g_2$ on $M_P$. Now, we define $\Psi:=\psi \circ \varphi^{-1}: M \to M$. Then, it is clear that $\Psi^* g'=g$. Moreover, since $\Psi^* \omega' =\omega$, it follows that $\Psi$ is a $\T^n$-invariant self-biholomorphism of $(M, J)$. In sum, we have completed the proof of Theorem \ref{thm:A}.

\subsubsection*{Proof of Theorem \ref{thm:B}}

Suppose $(M^{2n}, J, \T^n)$ is a complex orbifold with an effective and holomorphic $\T^n$-action. Moreover, we assume $(M, g,J, f)$ is a K\"ahler Ricci shrinker orbifold with $X=J\na f \in \lt$ and bounded Ricci curvature. 

Let $ \mathfrak{aut}^{X}(M)$ be the Lie algebra of real holomorphic vector fields on $M$ that commute with $JX$ and hence $X$, and let $\g^X$ be the Lie subalgebra of $\mathfrak{aut}^{X}(M)$ which is also Killing. The following Matsushima-type theorem is proved in \cite[Theorem 5.1]{CDS24} for the manifold case and can be generalized to the orbifold case by the same proof. Note that the only difference in the proof is the choice of the cutoff function, which, in the orbifold case, can be chosen as in the proof of Proposition \ref{prop:sca} as $\eta(f/A)$.

\begin{thm}[Matsushima theorem]
With the above assumptions, one has the decomposition:
\begin{align*}
\mathfrak{aut}^{X}(M) = \mathfrak{g}^{X} \oplus J \mathfrak{g}^{X}.
\end{align*}
\end{thm}

Furthermore, we can show as \cite[Proposition 5.6, Proposition 5.9]{CDS24} that $\mathfrak{aut}^{X}(M)$ is the Lie algebra of the finitely-dimensional Lie group $\mathrm{Aut}_0^X(M)$, which is the identity component of the group consists of all holomorphic automorphisms of $M$ that commute with the flow of $JX$. In addition, $\mathfrak{g}^{X}$ is the Lie algebra of a connected Lie group $G_0^X$, which is a maximal compact Lie subgroup of $\mathrm{Aut}_0^X(M)$.
Using the Matsushima theorem, we have the following Lemma similar to \cite[Lemma 4.1]{Ci22}.
\begin{lem}
\label{L602}
There exists a biholomorphism $\varphi: (M, J) \to (M, J)$ such that $\varphi^* g$ is $\T^n$-invariant.
\end{lem}
\begin{proof}
By our assumption, $[Y, X]=[Y, JX]=0$ for any $Y \in \lt$. Then, the holomorphic action of $\T^n$ gives $\T^n \subset \mathrm{Aut}_0^X(M)$. Then, $\T^n$ is contained in a maximal compact subgroup $G$ of $\mathrm{Aut}_0^X(M)$. By the Iwasawa's theorem \cite{Iw49}, all maximal compact subgroups are conjugate, and hence there exists $\varphi \in \mathrm{Aut}_0^X(M)$ so that $\varphi \circ \T^n \circ \varphi^{-1} \leq G_0^X$. Therefore, $\varphi^* g$ is $\T^n$-invariant.
\end{proof}
With Lemma \ref{L602} and Theorem \ref{thm:A}, the proof of Theorem \ref{thm:B} is straightforward.

\vskip10pt
Yu Li, Institute of Geometry and Physics, University of Science and Technology of China, No. 96 Jinzhai Road, Hefei, Anhui Province, 230026, China; Hefei National Laboratory, No. 5099 West Wangjiang Road, Hefei, Anhui Province, 230088, China; E-mail: yuli21@ustc.edu.cn; \\

Wenjia Zhang, School of Mathematical Sciences, University  of Science and Technology of China No. 96 Jinzhai Road, Hefei, Anhui Province, 230026, China; E-mail: wj12345678@mail.ustc.edu.cn\\


\begin{thebibliography}{99}
\bibitem{Ab01} M. Abreu, \emph{K\"ahler metrics on toric orbifolds}, J. Differential Geom. 58 (2001), 151–187.

\bibitem{Ab03} M. Abreu, \emph{K\"ahler geometry of toric manifolds in symplectic coordinates}, Fields Inst. Commun. 35 (2003), 1–24.

\bibitem{AT82} M. F. Atiyah, \emph{Convexity and commuting hamiltonians}, Bull. London Math. Soc. 14 (1982), 1-15.

\bibitem{Au04} M. Audin, \emph{Torus actions on symplectic manifolds}, volume 93 of Progress in Mathematics. Birkhäuser Basel, 2012.

\bibitem{Ap19} V. Apostolov, \emph{The K\"ahler geometry of toric manifolds}, CIRM-2019 Lecture Notes, arXiv:2208.12493.

\bibitem{ACGT04} V. Apostolov, D. M. J. Calderbank, P. Gauduchon, C. T\o{}nnesen-Friedman, \emph{Hamiltonian 2-forms in K\"ahler
geometry II Global classification}, J. Differential Geom. 68 (2004), 277-345.

\bibitem{BL97} L. Bates, E. Lerman, \emph{Proper group actions and symplectic stratified spaces}, Pacific 1. Math., 181(2),201-229.

\bibitem{Bam20} R. H. Bamler, \emph{Structure theory of non-collapsed limits of Ricci flow}, arXiv:2009.03243v2.

\bibitem{BCCD24} R. Bamler, C. Cifarelli, R. Conlon, A. Deruelle, \emph{A new complete two-dimensional shrinking gradient K\"ahler-Ricci soliton}, Geom. Funct. Anal. 34, 377–392 (2024).

\bibitem{BB13} R. J. Berman, B. Berndtsson, \emph{Real Monge-Amp\'ere equations and Kähler-Ricci solitons on toric log Fano varieties}, Ann. Fac. Sci. Toulouse Math. 22 (2013), 649-711.

\bibitem{Bor92} J. Borzellino, \emph{Riemannian geometry of orbifolds}, Ph.D. Thesis, UCLA (1992).

\bibitem{BGL08} D. Burns, V. Guillemin, E. Lerman, \emph{ Kähler metrics on singular toric varieties}, Pacific J. Math. 238(2008),
27-40.

\bibitem{BMP03} M. Boileau, S. Maillot, J. Porti, \emph{Three-dimensional orbifolds and their geometric structures}, Panoramas \& Synth\`eses, vol. 15, Soc. Math. France, Paris, 2003.


\bibitem{CCZ} H.-D. Cao, B.-L. Chen, X.-P. Zhu, \emph{Recent developments on Hamilton's Ricci flow}, Surveys in differential geometry, Vol. XII. Geometric flows, 47-112, Surv. Differ. Geom., 12, Int. Press, Somerville, MA, 2008.

\bibitem{CZ10} H.-D. Cao, D. Zhou, \emph{On complete gradient shrinking Ricci solitons}, J. Differ. Geom., 85 (2010), no. 2, 175-186.

\bibitem{C19} F. C. Caramello Jr., \emph{Introduction to Orbifolds}, arXiv.1909.08699v2. 

\bibitem{CBL07} B.-L. Chen, \emph{Strong uniqueness of the Ricci flow}, J. Differ. Geom., 82 (2009), no. 2, 363-382. 

\bibitem{CW12} X. Chen, B. Wang, \emph{Space of Ricci flows (I)}, Comm. Pure. Appl. Math. 65 (2012), no. 10, 1399-1457.

\bibitem{Ci22} C. Cifarelli, \emph{Uniqueness of shrinking gradient K\"ahler-Ricci solitons on non-compact toric manifolds}, J. of the London Math. Soc., Volume 106, Issue 4 p. 3746-3791.

\bibitem{CCD24} C. Cifarelli, R. J. Conlon, A. Deruelle, \emph{On finite time Type I singularities of the K\"ahler–Ricci flow on compact K\"ahler surfaces}. J. Eur. Math. Soc. (2024).

\bibitem{CDS24} R. J. Conlon, A. Deruelle, S. Sun, \emph{Classification results for expanding and shrinking gradient K\"ahler–Ricci
solitons}, Geom. Topol., 28(1):267–351, 2024.

\bibitem{CLS11} D. A. Cox, J. B. Little, H. K. Schenck, \emph{Toric varieties}, Grad. Stud. Math., vol. 124, Amer. Math. Society,
Providence, R.I., 2011.

\bibitem{De88} T. Delzant, \emph{Hamiltoniens p\'eriodiques et images convexes de l'application moment}, Bulletin de la Société Mathématique de France, 116 (3): 315–339.



\bibitem{GKRW17} F. Galaz-García, M. Kerin, M. Radeschi, and M. Wiemeler: \emph{Torus Orbifolds, Slice-Maximal Torus Actions, and Rational Ellipticity}, International Mathematics Research Notices, Vol. 2018, No. 18, pp. 5786-5822.

\bibitem{G94} V. Guillemin, \emph{Kaehler structures on toric varieties}, J. Differential Geom.40(1994),285–309.

\bibitem{GS82} V. Guillemin, S. Sternberg, \emph{Convexity properties of the moment mapping I}, Invent. Math. 67 (1982), 491-513.

\bibitem{GS84} V. Guillemin, S. Sternberg, \emph{A normal form forthe moment map}, Differential Geometrie Methods in Mathematieal Physies. S. Sternberg ed. Mathematical Physics Studies, 6. D. Reide1 Publishing Company.


\bibitem{Ha95} R. S. Hamilton, \emph{The formation of singularities in the Ricci flow}, Surveys in Differential Geom., 2(1995), 7-136, International Press. 



\bibitem{HM11} R. Haslhofer, R. M\"{u}ller, \emph{A compactness theorem for complete Ricci shrinkers}, Geom. Funct. Anal., 21 (2011), 1091-1116.


\bibitem{HNP94} J. Hilgert, K.-H. Neeb, and W. Plank,\emph{ Symplectic convexity theorems and coadjoint orbits}, Compos. Math. 94 (1994), 129–180.

\bibitem{Iw49} K. Iwasawa, \emph{On some types of topological groups}, Ann. of Math. (2) 50 (1949), 507–558.


\bibitem{KL14} B. Kleiner, J. Lott, \emph{Local Collapsing, Orbifolds, and Geometrization}, Ast\'erisque Monograph 365, 2014.

\bibitem{LT97} E. Lerman, S. Tolman, \emph{Hamiltonian torus actions on symplectic orbifolds and toric varieties}, Trans. Amer. Math. Soc. 349 (1997), 4201–4230.

\bibitem{LW20} Y. Li, B. Wang, \emph{Heat kernel on Ricci shrinkers}, Calc. Var. Partial. Differ. Equ. 59 (2020), article 194.

\bibitem{LW23} Y. Li, B. Wang, \emph{On K\"ahler Ricci shrinker surfaces}, arXiv:2301.09784.

\bibitem{Lott03} J. Lott, \emph{Some geometric properties of the Bakry-\'Emery-Ricci tensor}, Comment. Math. Helv. 78 (2003), no. 4, 865-883.

\bibitem{Ma84} C.-M. Marle, \emph{Le voisinage d'une orbite d'une action hamiltonienne d'un groupede Lie}, S\'eminaire Sud-Rhodanien de G\'eom\'etrie II (P. Dazord, N. Desolneux-Mouliseds.) pp. 19-35.

\bibitem{Ma85} C.-M. Marle, \emph{Mod\'ele d'action hamiltonienne d'un groupe the Lie sur une vari\'et\'e symplectique}, Rend. Sem. Mat. Univers. Politeen. Torino, 43(2), 227-251.

    \bibitem{Naber} A. Naber, \emph{Noncompact shrinking four solitons with nonnegative curvature}, J. Reine Angew. Math. 645(2010), 125-153. 

 \bibitem{NW} L. Ni, N. Wallach, \emph{On a classification of gradient shrinking solitons}, Math. Res. Lett. 15(2008), no. 5, 941-955. 

\bibitem{Pe02}  G. Perelman, \emph{The entropy formula for the Ricci flow and its geometric applications}, arXiv:math.DG/0211159.

\bibitem{PW10} P. Petersen, W. Wylie, \emph{On the classification of gradient Ricci solitons}, Geom. Topol. 14(2010), no. 4, 2277-2300.

\bibitem{Pr94} E. Prato, \emph{Convexity properties of the moment map for certain non-compact manifolds}, Comm. Anal. Geom. 2
(1994), 267–278.

\bibitem{PW94} E. Prato, S. Wu, \emph{Duistermaat-Heckman measures in a non-compact setting}, Compos. Math. 94 (1994), 113-128.

\bibitem{Sa56} I. Satake, \emph{On a generalization of the notion of manifold}, Proc. Natl. Acad. Sci. USA, 42(6) (1956), 359–363.


\bibitem{SZ11} Y. Shi, X. Zhu \emph{Kähler-Ricci solitons on toric Fano orbifolds}, Math. Zeit. (2012) 271:1241–1251.

\bibitem{Sj95} R. Sjamaar, \emph{Holomorphic slices, symplectic reduction andmultiplicities of representations}, Ann. of Math. (2) 141(1995), 87–129.



\bibitem{Th02} W. Thurston, \emph{The Geometry and Topology of Three-Manifolds}, Lecture notes (2002),

\bibitem{TZ02a} G. Tian, X. Zhu, \emph{Uniqueness of K\"ahler-Ricci solitons}, Acta Math. 184 (2000), 271–305.

\bibitem{TZ02b} G. Tian, X. Zhu, \emph{A new holomorphic invariant and uniqueness of Kähler-Ricci solitons}, Comment. Math. Helv. 77 (2002), 297-325.




\bibitem{Wu91} L.-F. Wu, \emph{The Ricci flow on 2-orbifolds with positive curvature}, J. Differential Geom. 33 (1991), no. 2, p. 575–596.

\bibitem{WW07} W. Wylie, \emph{Complete shrinking Ricci solitons have finite fundamental group}, Proc. Amer. Math. Soc.,136(5) (2007), 1803-1806.


\bibitem{Zh09} Z.-H. Zhang, \emph{On the completeness of gradient Ricci solitons}, Proc. Amer. Math. Soc.137(2009), no.8, 2755-2759.

\bibitem{Z07} G. M. Ziegler, \emph{Lectures on Polytopes}, Graduate Texts in Mathematics, 152, 2006.
\end{thebibliography}
\end{document}